\documentclass[10pt,reqno]{amsart}
\usepackage{hyperref}\hypersetup{colorlinks=true, citecolor=blue}
\usepackage{graphicx}
\usepackage{xfrac}
\usepackage{amsfonts,amsmath,amssymb,epsfig,mathrsfs,bm}
\usepackage{ifthen,color,mathtools}
\usepackage{epic,eepic}
\usepackage{esint,yfonts,stmaryrd}
\usepackage{a4wide}

\newtheorem{theorem}{Theorem}[section]
\newtheorem{lemma}[theorem]{Lemma}
\newtheorem{proposition}[theorem]{Proposition}
\newtheorem{corollary}[theorem]{Corollary}

\theoremstyle{definition}
\newtheorem{definition}[theorem]{Definition}

\theoremstyle{remark}
\newtheorem{remark}[theorem]{Remark}
\theoremstyle{remark}

\numberwithin{equation}{section}

\newcommand{\R}{\mathbb{R}}
\newcommand{\eps}{{\varepsilon}}

\newcommand{\dist}{{\textup {dist}}}

\newcommand{\de}{\partial}
\newcommand{\ph}{\varphi}

\renewcommand{\d}{{\rm d}}

\newcommand{\spt}{{\rm spt}\,}
\newcommand{\res}{\mathop{\hbox{\vrule height 7pt width .5pt depth 0pt
\vrule height .5pt width 6pt depth 0pt}}\nolimits}
\newcommand{\loc}{\textup{loc}}

\renewcommand{\and}{\quad \text{and} \quad}

\newcommand{\reg}{\textup{Reg}}
\newcommand{\sing}{\textup{Sing}}
\newcommand{\other}{\textup{Other}}
\renewcommand{\div}{\textup{div}}

\renewcommand{\top}{\textup{top}}
\newcommand{\low}{\textup{low}}
\newcommand{\sgn}{\textup{sign}}

\newcommand{\cL}{{\mathcal{L}}}
\newcommand{\cH}{{\mathcal{H}}}

\newcommand{\cT}{{\mathcal{T}}}

\newcommand\N{{\mathbb N}}

\newcommand{\ie}{\textit{i.e.}}
\newcommand{\eg}{\textit{e.g.}}
\newcommand{\ZZ}{\mathscr{Z}_{\varphi,\theta,\delta}}

\newcommand{\normphi}{\|\varphi\|}
\newcommand{\dm}{\d\mathfrak{m}}
\newcommand{\dmp}{\d\mathfrak{m}'}

\title[The free boundary of the fractional obstacle problem]
{
The local structure of the free boundary in the fractional obstacle problem}

\author[M.~Focardi]{Matteo Focardi}
\address{DiMaI, Universit\`a degli Studi di Firenze}
\curraddr{Viale Morgagni 67/A, 50134 Firenze (Italy)}
\email{matteo.focardi@unifi.it}
\author[E.~Spadaro]{Emanuele Spadaro}
\address{Dipartimento di Matematica, Universit\`a di Roma La Sapienza}
\curraddr{P.le Aldo Moro 5, 00185 Rome (Italy)}
\email{spadaro@mat.uniroma1.it}
\thanks{The authors have been partially funded
by the ERC-STG Grant n. 759229 HiCoS ``Higher Co-dimension Singularities: Minimal Surfaces and 
the Thin Obstacle Problem'' and by GNAMPA of  INdAM.}



\date{}
\begin{document}
\begin{abstract}
Building upon the recent results in \cite{FoSp17} we provide a thorough description of the
free boundary for solutions to the fractional obstacle problem in $\R^{n+1}$ with obstacle function 
$\varphi$ (suitably smooth and decaying fast at infinity) up to sets of null $\cH^{n-1}$ measure.
In particular, if $\varphi$ is analytic, the problem reduces to the zero obstacle case 
dealt with in \cite{FoSp17} and therefore we retrieve the same results:
\begin{itemize}
\item[(i)] local finiteness of the $(n-1)$-dimensional Minkowski content of the free boundary (and thus of its Hausdorff measure),
\item[(ii)] $\cH^{n-1}$-rectifiability of the free boundary,
\item[(iii)] classification of the frequencies and of 
the blow-ups up to a set of Hausdorff dimension at most $(n-2)$ in the free boundary.
\end{itemize}

Instead, if $\varphi\in C^{k+1}(\R^n)$, $k\geq 2$, similar results 
hold only for distinguished subsets of points in the free boundary 
where the order of contact of the solution with the obstacle function 
$\varphi$ is less than $k+1$.
\end{abstract}

\maketitle

%
%
\section{Introduction}

Quasi-geostrophic flow models \cite{CV}, anomalous diffusion in disordered media \cite{BG90} and American options with jump processes \cite{CT04} are some instances of constrained variational problems involving free boundaries for thin obstacle problems. 
In this paper we analyze the fractional obstacle problem with 
exponent $s\in(0,1)$, a problem that can be stated in several ways, each motivated by a different application and suited to be studied 
with different techniques. We follow here the variational approach: 
given $\varphi:\R^n\to\R$ smooth 
and decaying sufficiently fast at infinity, one seeks for minimizers of 
the $H^s$-seminorm 
\[
 [v]_{H^s}^2:=\int_{\R^n\times\R^n}\frac{|v(x')-v(y')|^2}{|x'-y'|^{n+2s}}\d x'\,\d y'
\]
$s \in (0,1)$, on the cone 
\begin{align*}
\mathscr{A}:=\Big\{v\in \dot{H}^s(\R^n):\, v(x')\geq \varphi(x')\Big\},
\end{align*}
where $\dot{H}^s(\R^n)$ is the homogeneous space defined as the closure in the
$H^s$ seminorm of $C_c^\infty(\R^n)$ functions.
Existence and uniqueness of a minimizer $w$ follow for all $s\in(0,1)$ if $n\geq 2$ 
(the case $n=1$ requires some care see \cite{Si07} and \cite{BaFiRo16}).
In addition, defining the fractional laplacian as 
\[
 (-\Delta)^s v(x'):= c_{n,s}\, \textrm{P.V.}\int_{\R^n}\frac{v(x')-v(y')}{|x'-y'|^{n+2s}}\d y',                       
\]
for $v\in \dot{H}^s(\R^n)$, the Euler-Lagrange conditions characterize $w$ as a 
distributional solution to the system of inequalities
 \begin{align}\label{e:ob-pb full space}
 \begin{cases}
 w(x') \geq \varphi(x') & \text{for }\; x'\in \R^n,\\
 (-\Delta)^s w(x') =0  & \text{for }\; w(x') > \varphi(x') ,\\
 (-\Delta)^s w(x') \geq0& \text{for }\; x'\in \R^n\,.
 \end{cases}
 \end{align}
 The most challenging regularity issues are then that of $w$ itself and that of its free boundary 
 \[
 \Gamma_\varphi(w):=\partial\big\{x'\in \R^n:\, w(x') =\varphi(x') \big\}.
 \]
 To investigate the fine properties of the solution $w$ of \eqref{e:ob-pb full space} 
 the groundbreaking paper by Caffarelli and Silvestre \cite{CaSi07} introduces
 an equivalent local counterpart for the fractional obstacle problem in terms 
 of the so called $a$-harmonic extension argument. 
 Indeed, it is inspired by the case $s=\sfrac 12$, in which it is nothing but the 
 harmonic extension problem. 
 More precisely, setting $a=1-2s$ for $s\in(0,1)$ and
 $\mathfrak{m} = |x_{n+1}|^a\cL^{n+1}$, 
 it turns out that any function $w$ satisfying 
\eqref{e:ob-pb full space} is the trace of a function $u\in H^1(\R^{n+1},\dm)$ 
solving for $x=(x',x_{n+1})\in \R^{n+1}$
\begin{equation}\label{e:ob-pb global}
\begin{cases}
u(x',0) \geq \varphi & \text{for }\; (x',0)\in \R^{n}\times\{0\},\\
u(x',x_{n+1}) = u (x', -x_{n+1}) &\text{for all}\; x\in \R^{n+1},\\
\div\big(|x_{n+1}|^a \nabla u(x)\big) = 0  & \text{for }\; 
x \in \R^{n+1} \setminus \big\{(x',0)\,:\,
u(x',0) = \varphi(x') \big\},\\
\div\big(|x_{n+1}|^a \nabla u(x)\big) \leq 0 & \text{in
}\,\mathscr{D}'(\R^{n+1}).
\end{cases}
\end{equation}
In particular, note that $u$ is unique minimizer of the Dirichlet energy
\[
\int_{\R^{n+1}} |\nabla \widetilde{v}|^2 |x_{n+1}|^a\d x
\]
on the class $\widetilde{\mathscr{A}}:=\big\{\widetilde{v}\in 
H^1(\R^{n+1},\dm):\, \widetilde{v}(x',0)\geq \varphi(x')\big\}$.
Viceversa, the trace $u(x',0)$ on the hyperplane $\{x_{n+1}=0\}$ of a solution 
$u$ to \eqref{e:ob-pb global} is a solution $w$ to \eqref{e:ob-pb full space}, 
as for all $x'\in\R^n$ (cf.~\cite{CaSi07})
\[
 \lim_{x_{n+1}\to 0^+}|x_{n+1}|^a \partial_{n+1} u(x)=-(-\triangle)^s u(x',0)\,.
\]

One then is interested into regularity issues for $u$ and for the corresponding 
free boundary $\Gamma_\varphi(u)$ (with a slight abuse of notation we use the 
same symbol as for the analogous set for $w$): the topological boundary, 
in the relative topology of $\R^n$, of the coincidence set of a solution $u$
\[
\Lambda_\varphi(u) := \big\{(x', 0) \in \R^{n+1}\,: u(x', 0) =\varphi(x') \big\}.
\]
The locality of the operator 
\begin{equation}\label{e:La}
L_a(v):=\div\big(|x_{n+1}|^a \nabla v(x)\big)
\end{equation}
in \eqref{e:ob-pb local} is the main advantage of the new formulation 
to perform the analysis of $\Gamma_\varphi(u)$. 
Indeed, being $\Gamma_\varphi(u)=\Gamma_\varphi(w)$ it permits the use of monotonicity 
ùand almost monotonicity type formulas analogous to those introduced 
by Weiss and Monneau for the classical obstacle problem 
(cf.~\cite{Caffa77, Caffa98, Weiss, Monneau03}).

Optimal interior regularity for $u$ has been established Caffarelli, Salsa and Silvestre 
in \cite[Theorem~6.7 and Corollary~6.8]{CaSaSi08} for any $s\in(0,1)$  (see also \cite{CaDSSa17}). 
The particular case $s=\sfrac12$ had been previously addressed by Athanasopoulos, Caffarelli 
and Salsa in \cite{AtCaSa08}.
Instead, despite all the mentioned progresses, the current picture for 
free boundary regularity theory is still incomplete. 
In this paper we go further on in this direction and deal with the non-zero obstacle case 
following the recent achievements obtained in the zero-obstacle case in \cite{FoSp17, FoSp17bis}.
Drawing a parallel with the theory in the zero-obstacle case, the free boundary 
$\Gamma_\varphi(u)$ can be split as a pairwise disjoint union of sets:
\begin{equation}\label{e:decomposizione free boundary}
\Gamma_\varphi(u) = \reg(u) \cup \sing(u)\cup \other(u),
\end{equation}
termed in the existing literature as the subset of regular, singular and nonregular/nonsingular points, respectively. 
These sets are defined via the infinitesimal behaviour of appropriate rescalings of the 
solution itself. More precisely, for $x_0 \in \Gamma_\varphi(u)$ a function ${\varphi}_{x_0}$ 
related to $\varphi$ can be conveniently defined (cf.  \eqref{e:rescaling0} and \eqref{e:rescaling phi}) in a way that  if 
\[
 u_{x_0,r}(y) :=
\frac{r^{\frac{n+a}{2}}\big(u(x_0+r\,y)-{\varphi}_{x_0}(x_0+r\,y)\big)}
{\Big(\int_{\de B_r} (u-{\varphi}_{x_0})^2\,|x_{n+1}|^{a}d\cH^{n}\Big)^{\sfrac12}}\; ,
\]
then the family of functions $\{u_{x_0,r}\}_{r>0}$ is 
pre-compact in $H^1_{\loc}(\R^{n+1},\dm)$ (see \cite[Section 6]{CaSaSi08}). 
The limits are called blowups of $u$ at $x_0$, they are 
homogeneous solutions of a fractional obstacle problem with
zero obstacle. The set of all such functions is denoted by 
$\textup{BU}(x_0)$. Their homogeneity $\lambda(x_0)$ depends 
only on the base point $x_0$ and not on the extracted subsequence, and it is called 
\textit{infinitesimal homogeneity} or \textit{frequency} of $u$ at $x_0$. 
It is indeed the limit value, as the radius vanishes, 
of an Almgren's type frequency function related to $u$ which
turns out to be non decreasing in the radius. Given this, one defines
\begin{gather*}
\reg(u):=\{x\in\Gamma_\varphi(u):\,\lambda(x_0)=1+s\},\quad
\sing(u):=\{x\in\Gamma_\varphi(u):\,\lambda(x_0)=2m,\,m\in\N\},\\
\other(u) := \Gamma_\varphi(u) \setminus \Big(\reg(u) 
\cup \sing(u)\Big).
\end{gather*}
According to the regularity of $\varphi$ different results are known in literature:
\begin{itemize}
 \item[(i)] Regular points: in \cite{CaSaSi08} for $\varphi\in C^{2,1}(\R^n)$ optimal 
 one-sided $C^{1,s}$ regularity of solutions is established. Moreover, $\reg(u)$ is
 shown to be locally a $C^{1,\alpha}$ submanifolf of codimension $2$
 in $\R^{n+1}$ (non-optimal regularity of the solution had been previously established in \cite{Si07});

 \item[(ii)] Singular points: for $\varphi$ analytic and $a=0$
 it is proved in \cite{GaPe09} 
 that $\sing(u)$ is $(n-1)$-rectifiable. The latter result
 has been very recently extended 
 to the full range $a\in(-1,1)$ and to
 $\varphi\in C^{k+1}(\R^n)$, $k\geq 2$, 
 in \cite{GaRos17}.
 Furthermore, fine properties of the singular set have been studied very recently 
 by Fern\'andez-Real and Jhaveri \cite{FJ18}.
\end{itemize}
It is also worth mentioning the paper by Barrios, Figalli and Ros-Oton \cite{BaFiRo16},
in which the authors study the fractional obstacle problem \eqref{e:ob-pb full space} with 
non zero obstacle $\varphi$ having compact support and satisfying suitable 
concavity assumptions. Under these assumptions, they are able to fully characterize
the free boundary, showing that $\other(u) = \emptyset$ and that at every
point of $\sing(u)$ the blowup is quadratic, \ie~the only admissible value
of $m$ is $1$. 
In addition, they are able to show that the singular set $\sing(u)$ is locally 
contained in a single $C^1$-regular submanifold (see also \cite{CaDSSa17} for 
the case of less regular obstacles and \cite{DaSa15,KPS15,KRS17,JN17} for higher 
regularity results on $\reg(u)$).

For ease of expositions we start with the simpler case in which the
obstacle function $\varphi$ is analytic, actually the slightly milder assumption 
\eqref{e:anal} below suffices (see Section~\ref{ss:cikappagamma} for related results in 
the case $\varphi\in C^{k+1}(\R^n)$). 
Indeed, after a suitable transformation (see Section~\ref{ss:analytic}) such a framework 
reduces to the zero obstacle case since in this setting $\widetilde{\varphi}$ turns out 
to be exactly the $a$-harmonic extension of $\varphi$.
Thus, in view of \cite[Theorems~1.1-1.3]{FoSp17} we may deduce the following result.
\begin{theorem}\label{t:phi analytic}
Let $u$ be a solution to the fractional obstacle problem 
\eqref{e:ob-pb global} with obstacle function $\varphi:\R^n\to\R$ such that 
\begin{equation}\label{e:anal}
\{\varphi>0\}\subset\subset\R^n,\quad 
\textrm{$\varphi$ is real analytic on $\{\varphi>0\}$}.
\end{equation}
Then, 
\begin{itemize}
 \item[(i)] the free boundary $\Gamma_\varphi(u)$ has finite 
 $(n-1)$-dimensional Minkowski content: more precisely, there exists a constant $C>0$ such that
\begin{equation}\label{e:misura1}
\cL^{n+1}\big(\cT_r(\Gamma_\varphi(u))\big) \leq C\,r^2
\quad \forall\; r \in (0,1),
\end{equation}
where $\cT_r(\Gamma_\varphi(u)) := \{ x \in \R^{n+1}: \dist(x,\Gamma(u)) <r\}$;

\item[(ii)] the free boundary $\Gamma_\varphi(u)$ is $(n-1)$-rectifiable, 
\ie~there exist at most countably many  $C^1$-regular submanifolds 
$M_i \subset \R^n$ of dimension $n-1$ such that 
\begin{equation}\label{e:rect1}
\cH^{n-1}\big(\Gamma_\varphi(u) \setminus \cup_{i\in\N} M_i\big) = 0.
\end{equation}
\end{itemize}
Moreover, there exists a subset $\Sigma(u)\subset \Gamma_\varphi(u)$ with Hausdorff 
dimension at most $n-2$ such that for every $x_0\in \Gamma_\varphi(u)\setminus \Sigma(u)$ 
the infinitesimal homogeneity $\lambda$ of $u$ at $x_0$ belongs to 
$\{2m, 2m-1+s, 2m+2s\}_{m\in\N\setminus\{0\}}$.
\end{theorem}

The analysis is more involved in case $\varphi$ is not analytic,
since one cannot in principle avoid contact points of infinite order
between the solution and the obstacle, and the free boundary
can be locally an arbitrary compact set $K\subset \R^n$ (explicit examples are provided in \cite{FR19}).
In view of this, we follow the existing literature and we 
consider only those points in the free boundary in which $u$ has order of contact with 
$\varphi$ less than $k+1$:
given $u$ a solution to the fractional obstacle problem 
\eqref{e:ob-pb global} and given a constant $\theta\in(0,1)$ we set
\begin{equation}\label{e:Hgrowth intro}
\Gamma_{\varphi,\theta}(u)
:=\big\{x_0\in \Gamma_\varphi(u):\, \liminf_{r\downarrow 0}
r^{-(n+a+2(k+1-\theta))}H_{u_{x_0}}(r)>0 \big\},
\end{equation}
where $H_{u_{x_0}}$ is defined in \eqref{e:Hux0} and it is 
related to the $L^2(\partial B_r,\dmp)$ norm of
$u_{x_0,r}$
(cf. Section~\ref{ss:cikappagamma} for more details).
For this subset of points of the free boundary we can still prove some of the results stated in Theorem~\ref{t:phi analytic}.
\begin{theorem}\label{t:phi Ck}
Let $u$ be a solution to the fractional obstacle problem 
\eqref{e:ob-pb global} with obstacle function 
$\varphi\in C^{k+1}(\R^n)$, $k\geq 2$, and let $\theta\in(0,1)$. 
Then, $\Gamma_{\varphi,\theta}(u)$ is $(n-1)$-rectifiable.
%
Moreover, there exists a subset $\Sigma_\theta(u)\subset \Gamma_{\varphi,\theta}(u)$ 
with Hausdorff dimension at most $n-2$ such that for every $x_0\in 
\Gamma_{\varphi,\theta}(u)\setminus \Sigma_\theta(u)$ 
the infinitesimal homogeneity $\lambda$ of $u$ at $x_0$ belongs to 
$\{2m, 2m-1+s, 2m+2s\}_{m\in\N\setminus\{0\}}$.
\end{theorem}

This note extends the results of \cite{FoSp17} to the case of nonconstant obstacles.
It is clear by the examples of arbitrary compact sets as contact sets of suitable 
solutions of the problem, that in general the free boundary for nonconstant smooth 
obstacles does not possess any structure and that the key ingredient for the analysis 
of the free boundary is the \textit{analiticity} of the obstacles as shown in Theorem~\ref{t:phi analytic}.
Nevertheless, for a subset of the free boundary, characterized as those points of 
finite order of contact, e.g. the points $\Gamma_{\varphi,\theta}$ already considered 
in the literature (see \cite{GaRos17}), a partial regularity still holds even in 
the framework of non analytic obstacles, as proven in Theorem~\ref{t:phi Ck}.
The main novelty of this paper with respect to \cite{FoSp17} consists in the analysis 
of the spatial dependence of the frequency for nonconstant obstacles: indeed, in this case 
the frequency is defined differently from point to point, by taking into account the geometry 
of the obstacle itself. It is not at all evident to which extent the oscillation of the frequency can be controlled. 
The results of Section~4 show that this kind of estimates are not completely rigid and extend to nonflat obstacles. 
Hence, this paper contributes to the program of broadening the results initially proven for the Signorini problem 
with zero obstacles to the case of the obstacle problem for the fractional Laplacian (see e.g. \cite{AtCaSa08, GaRos17}), 
providing a generalization of the known results on the structure of the free boundary firstly proven in \cite{FoSp17}.

%
%

%

\section{Analytic obstacles}

In this section we deal with analytic obstacles. 
We report first on some results related to the Caffarelli-Silvestre $a$-harmonic 
extension argument that will be instrumental to reduce the analytic type fractional 
obstacle problem to the lower dimensional obstacle problem. 
We provide then the proof of Theorem~\ref{t:phi analytic}. 

\subsection{Extension results}\label{ss:analytic}

We start off stating a lemma in which it is proved that there exists a canonical $a$-harmonic extension 
of a polynomial which is a polynomial itself (see \cite[Lemma~5.2]{GaRos17}).
We denote by $\mathscr{P}_l(\R^n)$ the finite dimensional vector space of homogeneous polynomials
of degree $l\in \N$.

\begin{lemma}\label{l:grufolo polinomio}
For every $l\in \N$, there exists a unique linear extension operator 
$\mathscr{E}_l:\mathscr{P}_l(\R^n)\to\mathscr{P}_l(\R^{n+1})$
such that for every $p\in \mathscr{P}_l(\R^n)$ we have
 \[
  \begin{cases}
   -\div(|x_{n+1}|^a\nabla\mathscr{E}_l[p])=0  & \text{in }\,\mathscr{D}'(\R^{n+1}),\cr
   \mathscr{E}_l[p](x',0)=p(x') & \textrm{for all $x'\in\R^n$}, \cr
   \mathscr{E}_l[p](x',-x_{n+1})= \mathscr{E}_l[p](x',x_{n+1}) & 
   \textrm{for all $x\in \R^{n+1}$}.
  \end{cases}
 \]
\end{lemma}
\begin{proof}
Let $p\in \mathscr{P}_l(\R^n)$ and set
\[
\mathscr{E}_l[p](x',x_{n+1}) := \sum_{j=0}^{\lfloor\sfrac{l}{2}\rfloor}
p_{2j}(x') \,x_{n+1}^{2j},
\]
with $p_{2j}(x'):= -\frac{1}{4j(j-s)} \triangle p_{2j-2}(x')$ if $j\in\{1,\ldots,\lfloor\sfrac{l}{2}\rfloor\}$ 
and $p_0 := p$. It is then easy to verify that $\mathscr{E}_l$ satisfies all the stated properties.
\end{proof}
\begin{remark}\label{r:continuity ext}
 In particular, $\mathscr{E}_l$ is a continuous operator, $l\in\N$. 
 We will use in what follows that there exists a constant $C=C(n,l)>0$ 
 such that for every $p\in \mathscr{P}_l(\R^n)$ and for every $r>0$
\begin{equation}\label{e:continuity ext}
\|\mathscr{E}_l[p]\|_{L^\infty(B_r)}\leq C\,\|p\|_{L^\infty(B_r')}.
\end{equation}
\end{remark}

We provide next the main result that reduces locally the analytic 
case to the zero obstacle case (cf. \cite[Lemma~5.1]{GaRos17}).
\begin{lemma}\label{l:grufolo}
 Let $\varphi:\Omega\to\R$ be analytic, $\Omega\subset\R^n$ open. Then for all $K\subset \subset\Omega\times\{0\}$ there exists $r>0$
 such that, for every $x_0 \in K$, there exists a function 
 $\mathscr{E}_{x_0}[\varphi]:\,B_r(x_0)\to\R$ such that 
 \begin{itemize}
  \item[(i)]  $-\div(|x_{n+1}|^a\nabla \mathscr{E}_{x_0}[\varphi])=0$ in $\mathscr{D}'\big(B_r(x_0)\big)$;
  \item[(ii)] $\mathscr{E}_{x_0}[\varphi](x',0)=\varphi(x')$ $\forall\;(x',0)\in B_r(x_0)$;
  \item[(iii)] $\mathscr{E}_{x_0}[\varphi]$ is analytic in $B_r(x_0)$.
 \end{itemize}
\end{lemma}
\begin{proof}
For every $x_0$ as in the statement, we can locally expand $\varphi$ in power series as
$\varphi(x') = \sum_{\alpha} c_\alpha (x'-x_0)^\alpha$.
Then, we set $\mathscr{E}_{x_0}[\varphi](x) 
:= \sum_{\alpha} c_\alpha \mathscr{E}_{|\alpha|}[p_\alpha](x-x_0)$ where $p_\alpha(x'):= (x')^\alpha$.
From the explicit formulas in the proof of Lemma~\ref{l:grufolo polinomio} it is easily
verified that the power series defining $\mathscr{E}_{x_0}[\varphi]$ is converging in $B_r(x_0)$
and gives an analytic $a$-harmonic extension in $B_r(x_0)$ with $r>0$ uniform on compact sets.
\end{proof}

\subsection{Proof of Theorem~\ref{t:phi analytic}}

Theorem~\ref{t:phi analytic} follows straightforwardly 
from \cite[Theorems~1.1-1.3]{FoSp17}. 
As explained in the introduction $w(x') = u(x',0)$ solves the fractional obstacle problem \eqref{e:ob-pb full space}. 
By the maximum principle $u(x',0) >  0$ for all $x'\in\R^n$. 
Therefore, $\Gamma_\varphi(u)\subset \{\varphi>0\}\subset\subset\R^n$.
Let $r>0$ be the radius in Lemma~\ref{l:grufolo} corresponding to the compact set $\Gamma_\varphi(u)$.
By compactness we  cover $\Gamma_\varphi(u)$ with a finite number of balls 
$B_r(x_i)$, with $x_i\in\R^n\times\{0\}$. In each ball $B_r(x_i)$ we consider
the corresponding function $u-\mathscr{E}_{x_i}[\varphi]$, with 
$\mathscr{E}_{x_i}[\varphi]$ provided by Lemma~\ref{l:grufolo}, and note that 
it solves a zero lower dimensional obstacle problem \eqref{e:ob-pb global}. Hence, we can conclude by
the quoted \cite[Theorems~1.1-1.3]{FoSp17}.

\section{$C^{k+1}$ obstacles}\label{ss:cikappagamma}

In this section we deal with the more demanding case of $C^{k+1}$ obstacles, $k\geq 2$. 
It is convenient to reduce the analysis of \eqref{e:ob-pb global} 
to that of the following localized problem
\begin{equation}\label{e:ob-pb local}
\begin{cases}
u(x',0) \geq \varphi(x') & \text{for }\; (x',0)\in B_1',\\
u(x',x_{n+1}) = u (x', -x_{n+1}) &\text{for }\; x=(x',x_{n+1})\in B_1,\\
\div\big(|x_{n+1}|^a \nabla u(x)\big) = 0  & \text{for }\; 
x \in B_1 \setminus \big\{(x',0)\in B_1'\,:\,
u(x',0) = \varphi(x') \big\},\\
\div\big(|x_{n+1}|^a \nabla u(x)\big) \leq 0 & 
\text{in }\, 
 \mathscr{D}'(B_1)\,,\\
\end{cases}
\end{equation}
for $\varphi\in C^{k+1}(B_1')$. 
In what follows, we shall assume that 
$\normphi_{C^{k+1}(B_1')}\leq 1$. This assumption can be 
easily matched by a simple scaling argument (cf. the proof of 
Theorem \ref{t:phi Ck}).

For any $x_0\in B_1'$ we denote by $T_{k,x_0}[\varphi]$ 
the Taylor polynomial of $\varphi$ of order $k$ at $x_0$: 
\[
 T_{k,x_0}[\varphi](x'):=\sum_{|\alpha|\leq k}\frac{D^\alpha\varphi(x_0)}{\alpha!}p_\alpha(x'-x_0),
\]
where $\alpha=(\alpha_1,\ldots,\alpha_n)\in\N^n$, $D^\alpha=\partial_{x_1}^{\alpha_1}\cdots\partial_{x_n}^{\alpha_n}$,
$p_\alpha(x'):=(x')^\alpha=x_1^{\alpha_1}\cdots x_n^{\alpha_n}$, $|\alpha|:=\alpha_1+\ldots+\alpha_n$
and $\alpha!:=\alpha_1!\cdots\alpha_n!$. 
In what follows we will repeatedly use that (recall that $\normphi_{C^{k+1}(B_1')}\leq 1$)
\begin{equation}\label{e:diff pol funz}
 |T_{k,x_0}[\varphi](x')-\varphi(x')|\leq \frac1{(k+1)!}
 |x'-x_0|^{k+1},
\end{equation}
and 
\begin{equation}\label{e:diff pol funz bis}
 |T_{k,x_0}[\partial_e\varphi](x')-\partial_e\varphi(x')|\leq 2 |x'-x_0|^k
\end{equation}
for all unit vectors $e\in\R^{n+1}$ such that $e\cdot e_{n+1}=0$.

Let then $\mathscr{E}\big[T_{k,x_0}[\varphi]\big]$ be the $a$-harmonic extension of $T_{k,x_0}[\varphi]$, namely
\[
 \mathscr{E}\big[T_{k,x_0}[\varphi]\big](x):=\sum_{|\alpha|\leq k}\frac{D^\alpha\varphi(x_0)}{\alpha!}
 \mathscr{E}_{|\alpha|}[p_\alpha(\cdot-x_0)](x),
\]
where $\mathscr{E}_l$ are the extension operators in Lemma~\ref{l:grufolo polinomio}. 
By the translation invariance of the operator, we point out that 
\begin{equation}\label{e:E traslato}
 \mathscr{E}_{|\alpha|}[p_\alpha(\cdot-x_0)](x)=\mathscr{E}_{|\alpha|}[p_\alpha](x-x_0).
\end{equation}
Set
\begin{equation}\label{e:phitilde}
{\varphi}_{x_0}(x):=\varphi(x')-T_{k,x_0}[\varphi](x')
+\mathscr{E}\big[T_{k,x_0}[\varphi]\big](x), 
\end{equation}
and
\begin{equation}\label{e:vx0}
 u_{x_0}(x):=u(x)-{\varphi}_{x_0}(x).
\end{equation}
Recalling that $\mathscr{E}\big[T_{k,x_0}[\varphi]\big](x',0)=T_{k,x_0}[\varphi](x')$, then 
$\Lambda_\varphi(u)=\{(x',0)\in B_1':\,u_{x_0}(x',0)=0\}$, 
and thus in particular $\Gamma_\varphi(u)=\partial_{B_1'}\{(x',0)\in B_1':\,u_{x_0}(x',0)=0\}$,
where $\partial_{B_1'}$ is the relative boundary in the hyperplane $\{x_{n+1}=0\}$. 
We note that $u_{x_0}$ is not a solution of a fractional obstacle problem as in \eqref{e:ob-pb local} 
with null obstacle, but rather of a related obstacle problem with drift as discussed in what follows
(cf. \eqref{e:funz rescaling}).

First, from the regularity assumption on $\varphi$, from Lemma~\ref{l:grufolo polinomio} and from 
estimate \eqref{e:diff pol funz} we infer that $L_a({\varphi}_{x_0})$ is a function in $L^1(B_1)$ (recall 
the definition of the operator $L_a$ given in \eqref{e:La}). 
Moreover, estimate \eqref{e:diff pol funz} gives for all $x\in B_1\setminus B_1'$ 
\begin{align}\label{e:Laphi}
 |L_a({\varphi}_{x_0}(x))|
 &=|\div\big(|x_{n+1}|^a\nabla(\varphi-T_{k,x_0}[\varphi])(x')\big)|\notag\\
 &=|x_{n+1}|^a|\triangle(\varphi(x')-T_{k,x_0}[\varphi](x'))|\leq 
 |x_{n+1}|^a|x'-x_0|^{k-1}.
\end{align}
In turn, this yields that the distribution $L_a(u_{x_0})$ is given by the sum of 
a function in $L^1(B_1)$ and of a non-positive measure supported on $B_1'$, namely,
\begin{equation}\label{e:Lav}
L_a(u_{x_0}(x))=\div\big(|x_{n+1}|^a \nabla u(x)\big)-L_a({\varphi}_{x_0}(x))\cL^n\res B_1\,.
\end{equation}

The following result resumes the regularity theory developed by Caffarelli, Salsa 
and Silvestre in \cite[Proposition~4.3]{CaSaSi08}. 
\begin{theorem}\label{t:reg}
Let $u$ be a solution to the fractional obstacle problem \eqref{e:ob-pb local} 
in $B_1$, $x_0\in B_r'$, $r\in(0,1)$, then $u_{x_0}\in C^{0,\min\{2s,1\}}(B_{1-r}(x_0))$, 
$\de_{x_i} u_{x_0}\in C^{0,s}(B_{1-r}(x_0))$ for $i=1,\ldots, n$, and
$|x_{n+1}|^a\de_{x_{n+1}} u_{x_0} \in C^{0,\alpha}(B_{1-r}(x_0))$ 
for all $\alpha\in(0,1-s)$. Moreover, there exists a constant 
$C_{\ref{t:reg}}=C_{\ref{t:reg}}(n,a,\alpha,r)
>0$ such that
\begin{align}\label{e:C1_1/2 reg}
&\|u_{x_0}\|_{C^{0,\min\{2s,1\}}(B_{\frac{1-r}{2}}(x_0))}+ 
\|\nabla' u_{x_0}\|_{C^{0,s}(B_{\frac{1-r}{2}}(x_0);R^n)}\notag\\&+ 
\|\sgn(x_{n+1})\,|x_{n+1}|^a\de_{x_{n+1}} u_{x_0}\|_{C^{0,\alpha}(B_{\frac{1-r}{2}}(x_0))}
\leq C_{\ref{t:reg}}\, \|u_{x_0}\|_{L^2(B_{1-r}(x_0), \dm)},
\end{align}
where $\nabla' u_{x_0} = (\de_{x_1}u_{x_0}, \ldots, \de_{x_n}u_{x_0})$
is the horizontal gradient.
\end{theorem}

In particular, the function $u$ is analytic in $\big\{x_{n+1} >0\big\}$
(see, \eg, \cite{Hopf32}) and the following boundary conditions holds:
%
\begin{align}
\lim_{x_{n+1}\downarrow 0^+} x_{n+1}^a \de_{n+1} u(x',x_{n+1}) = 0
& \quad \text{for }\; x'\in B_1'\setminus \Lambda_{\varphi}(u),\label{e:bd condition1}\\
\lim_{x_{n+1}\downarrow 0^+} x_{n+1}^a \de_{n+1} u(x',x_{n+1}) \leq 0
& \quad \text{for }\; x'\in B_1'.\label{e:bd condition2}
\end{align}
In particular,
\begin{equation}\label{e:bd distribution}
(u(x',0)-\varphi(x'))\,\lim_{x_{n+1}\downarrow 0^+} x_{n+1}^a \de_{n+1} u(x',x_{n+1}) = 0 
\quad \text{for }\; x'\in B_1'. 
\end{equation}
Furthermore, for $B_r(x_0)\subset B_1$ and $x_0\in B_1'$,
an integration by parts implies that 
\begin{align}\label{e:funz rescaling0}
 \int_{B_r(x_0)}&|\nabla u|^2|x_{n+1}|^a\d x
 -\int_{B_r(x_0)}|\nabla \varphi_{x_0}|^2|x_{n+1}|^a\d x\notag\\
 &= \int_{B_r(x_0)}|\nabla u_{x_0}|^2|x_{n+1}|^a\d x
+2 \int_{B_r(x_0)}\nabla u_{x_0}\cdot\nabla \varphi_{x_0}|x_{n+1}|^a\d x\notag\\
&=\int_{B_r(x_0)}|\nabla u_{x_0}|^2|x_{n+1}|^a\d x
-2\int_{B_r(x_0)}u_{x_0} L_a(\varphi_{x_0})\,\d x
+2 \int_{\partial B_r(x_0)}u_{x_0} 
\partial_{\nu}\varphi_{x_0}|x_{n+1}|^a\d x\,,
\end{align}
where in the second equality we have used that 
$\mathscr{E}[T_{k,x_0}(\varphi)]$ is even with respect to the hyperplane $\{x_{n+1}=0\}$ to deduce that
\[
 \lim_{x_{n+1}\to 0}
\partial_{n+1} \varphi_{x_0}(x) |x_{n+1}|^a=0.
\]
In particular, since the last addend in \eqref{e:funz rescaling0} only depends on the boundary values of
$u_{x_0}$, it follows that $u_{x_0}$ is a minimizer of the functional 
\begin{align}\label{e:funz rescaling}
\int_{B_r(x_0)}|\nabla v|^2|x_{n+1}|^a\d x
-2\int_{B_r(x_0)}v L_a(\varphi_{x_0})\,\d x
\end{align}
among all functions $v\in u_{x_0}+H^1_0(B_r(x_0),\dm)$ and satisfying
$v(x',0)\geq 0$ on $B_r'(x_0)$. 
Equivalently, we will say that $u_{x_0}$ is a local minimizer of the functional
in \eqref{e:funz rescaling} subject to null obstacle conditions.

\begin{remark}\label{r:rescal} 
We record here some bounds that shall be employed extensively 
in what follows. By using the linearity and continuity of the extension 
operator $\mathscr{E}_k$ (cf. Remark \ref{r:continuity ext}), together 
with estimate \eqref{e:diff pol funz} we get for all $z\in B_r$
\begin{align}\label{e:stima rescal}
|u_{x_0}(z)&-u_{x_1}(z)|  =|\varphi_{x_0}(z)-\varphi_{x_1}(z)|\notag\\
&\leq \big|T_{k,x_0}[\varphi](z')-T_{k,x_1}[\varphi](z')\big|
+\big|\mathscr{E}\big[T_{k,x_0}[\varphi]\big](z)-\mathscr{E}\big[T_{k,x_1}[\varphi]\big](z)\big|\notag\\
&\leq \|T_{k,x_0}[\varphi]-T_{k,x_1}[\varphi]\|_{L^\infty(B_r')}
+\|\mathscr{E}\big[T_{k,x_0}[\varphi]\big]-\mathscr{E}\big[T_{k,x_1}[\varphi]\big]\|_{L^\infty(B_r)}\notag\\ &
\stackrel{\eqref{e:continuity ext}}{\leq} C\,\|T_{k,x_0}[\varphi]-T_{k,x_1}[\varphi]\|_{L^\infty(B_r')}
\leq C\big(\|\varphi-T_{k,x_0}[\varphi]\|_{L^\infty(B_r')}+
\,\|\varphi-T_{k,x_1}[\varphi]\|_{L^\infty(B_r')}\big)\notag\\
& \stackrel{\eqref{e:diff pol funz}}{\leq} C\,
\big(\max_{z\in B_r'}|z-x_0|^{k+1}+\max_{z\in B_r'}|z-x_1|^{k+1}\big)\,.
\end{align}
for some constant $C=C(n,a,k)>0$. 
Since $\nabla\big(T_{k,x_i}[\varphi]\big)=T_{k-1,x_i}[\nabla\varphi]$, $i\in\{0,1\}$, arguing as above, 
using \eqref{e:diff pol funz bis} rather than \eqref{e:diff pol funz}, we conclude that
\begin{align}\label{e:stima grad rescal}
\big|\nabla(u_{x_0}(z) - u_{x_1}(z))| &=\big|\nabla(\varphi_{x_0}(z)-\varphi_{x_1}(z))\big|
\leq C\,
\big(\max_{z\in B_r'}|z-x_0|^{k}+\max_{z\in B_r'}|z-x_1|^{k}\big)\,,
\end{align}
for some constant $C=C(n,a,k)>0$. 
\end{remark}

\subsection{A frequency type function}\label{s:frequency}

Building upon the approach developed in \cite{FoSp17} we consider
a quantity strictly related to \textit{Almgren's frequency function}
and instrumental for developing the free boundary analysis in the subsequent
sections. Let $\phi:[0,+\infty) \to [0,+\infty)$ be defined by
\[
\phi(t) := 
\begin{cases}
1 & \text{for } \; 0\leq t \leq \frac{1}{2},\\
2\,(1-t) & \text{for } \; \frac{1}{2} < t \leq 1,\\
0 & \text{for } \; 1 < t ,\\
\end{cases}
\]
then given the solution $u$ to \eqref{e:ob-pb local}, a point 
$x_0 \in B_1'$ and the corresponding function $u_{x_0}$ in 
\eqref{e:vx0}, we define for all $0<r<1-|x_0|$
\[
I_{u_{x_0}}(r) :=\frac{r\, G_{u_{x_0}}(r)}{H_{u_{x_0}}(r)}
\]
where
\[
 G_{u_{x_0}}(r):=  - \frac 1r\int \dot{\phi}\Big(\textstyle{\frac{|x-x_0|}{r}}\Big)
\,{u_{x_0}}(x)\nabla {u_{x_0}}(x)\cdot\frac{x-x_0}{|x-x_0|}\,|x_{n+1}|^a\,\d x,
\]
and
\begin{equation}\label{e:Hux0}
H_{u_{x_0}}(r) :=- \int \dot{\phi}\Big(\textstyle{\frac{|x-x_0|}{r}}\Big)
\frac{u_{x_0}^2(x)}{|x-x_0|}\,|x_{n+1}|^a\,\d x.
\end{equation}
Here $\dot{\phi}$ indicates the derivative of $\phi$.
Clearly, $I_{u_{x_0}}(r)$ is well-defined as long as $H_{u_{x_0}}(r)>0$, in what 
follows when writing $I_{u_{x_0}}(r)$ we shall tacitly assume that the latter 
condition is satisfied.

For later convenience, we introduce also the notation
\[
D_{u_{x_0}}(r) := \int \phi\big(\textstyle{\frac{|x-x_0|}{r}}\big)\,
|\nabla {u_{x_0}}(x)|^2\,|x_{n+1}|^a\d x,
\]
and 
\[
E_{u_{x_0}}(r):= \int -\dot{\phi}\Big(\textstyle{\frac{|x-x_0|}{r}}\Big)
\frac{|x-x_0|}{r^2}\Big(\nabla {u_{x_0}}(x) \cdot
\frac{x-x_0}{|x-x_0|}\Big)^2\,|x_{n+1}|^a\, \d x.
\]
In particular, note that for all $r>0$ 
\begin{equation}\label{e:Cauchy Schwarz}
 H_{u_{x_0}}(r) \; E_{u_{x_0}}(r) - G_{u_{x_0}}^2(r)\geq 0
\end{equation}
by Cauchy-Schwarz inequality.


 \begin{remark}
 In case $\varphi=0$, then $u_{x_0}= u$ for all $x_0\in B_1'$ and 
 $G_u=D_u$. Thus, $I_{u_{x_0}}$ boils down to the variant of Almgren's 
 frequency function used in \cite{FoSp17}.
 \end{remark}
 \begin{remark}\label{r:scaling}
 If $u$ is a solution to the fractional obstacle problem \eqref{e:ob-pb local}, 
 then for every  $c>0$, $x_0\in B_1'$ and $r >0$ such that $B_r(x_0)\subset B_1$, 
 the function $\hat{u}(y) := c\,u(x_0 +r\,y)$ solves \eqref{e:ob-pb local} on $B_1$ 
 with obstacle function $\hat{\varphi}(y):= c\,\varphi(x_0 +r\,y)$. 
 Therefore, if $x_1=x_0+ry_1\in B_1'$ we have 
 $T_{k,y_1}[\hat{\varphi}](y')=c\,T_{k,x_1}[\varphi](x_0+ry')$ 
 and $\hat{u}_{y_1}(y)=c\,u_{x_1}(x_0+ry)$. Thus,
 $I_{\hat{u}_{y_1}}(\rho) = I_{u_{x_1}}(\rho\,r)$ for every $\rho \in (0,1)$.

In particular, this shows that the frequency function is scaling invariant, 
in the sequel we will use this property repeatedly. 
 \end{remark}

\subsection{Almost monotonicity of $I_{u_{x_0}}$ at distinguished points}

In this subsection we prove the quais-monotonicity of $I_{u_{x_0}}$ 
for a suitable subset of points of the free boundary.
We prove first some useful identities in a generic point $x_0$ of 
$B'_1$.
\begin{lemma}
Let $u$ be a solution to the fractional obstacle problem \eqref{e:ob-pb local} in $B_1$.
Then, for all $x_0 \in B'_1$ and $t\in(0,1-|x_0|)$, it holds
\begin{gather}
D_{u_{x_0}}(t)=G_{u_{x_0}}(t)
-\int \phi\big(\textstyle{\frac{|x-x_0|}{t}}\big)u_{x_0}(x)L_a(u_{x_0}(x))\d x
\label{e:D}\\
\frac{\d}{\d t}\big(H_{u_{x_0}}(t)\big)
=\frac{n+a}{t}\,H_{u_{x_0}}(t)+2\,G_{u_{x_0}}(t)
\label{e:Hprime} \\
\frac{\d}{\d t}\big(D_{u_{x_0}}(t)\big)
=\frac{n+a-1}{t} \, 
D_{u_{x_0}}(t)+2\,E_{u_{x_0}}(t)
-\frac 2t\,\int \phi\big(\textstyle{\frac{|x-x_0|}{t}}\big)\nabla u_{x_0}(x)\cdot(x-x_0)
 L_a\big(u_{x_0}(x)\big)\d x.\label{e:Dprime}
\end{gather}
\end{lemma}
\begin{remark}
With an abuse of notation, the integration in the last addends in \eqref{e:D} and \eqref{e:Dprime} is meant with respect 
to the reference measure $L_a(u_{x_0})$. Actually, we use this notation because from the proofs of \eqref{e:D} and 
\eqref{e:Dprime} it turns out that one can consider equivalently its absolutely continuous part $L_a(\varphi_{x_0})$. 
\end{remark}
\begin{proof}
To show \eqref{e:D}, \eqref{e:Hprime} and \eqref{e:Dprime},
we assume without loss of generality that $x_0 = \underline{0}$. 

For \eqref{e:D} we consider the vector field
$V(x) := \phi\big(\textstyle{\frac{|x|}{t}}\big)\,{u_{\underline{0}}}(x)\,\nabla {u_{\underline{0}}}(x)\,|x_{n+1}|^a$. 
Clearly, $V$ has compact support and $V \in C^{\infty}(B_1\setminus B_1',\R^{n+1})$.
Moreover, for $x_{n+1}\neq 0$
\[
V(x)\cdot e_{n+1}=\phi\big(\textstyle{\frac{|x|}{t}}\big)\,{u_{\underline{0}}}(x)\,
\partial_{n+1} u_{\underline{0}}(x)\,|x_{n+1}|^a\,,
\]
so that $\lim_{y\downarrow (x',0^+)}V(y)\cdot e_{n+1}=0$. 
 Indeed, recalling that $\mathscr{E}[T_{k,x_0}(\varphi)]$ is even with respect to the hyperplane 
 $\{x_{n+1}=0\}$ (cf. Lemma~\ref{l:grufolo polinomio}): if $(x',0)\in\Lambda_\varphi(u)$ we exploit the regularity of $u$ resumed in 
 Theorem~\ref{t:reg} to conclude; instead, if 
 $(x',0)\notin\Lambda_\varphi(u)$ it suffices to use \eqref{e:bd condition1}. 
Thus, the distributional divergence of $V$ is the $L^1(B_1)$ function given by
\begin{align*}
\div\, V (x) & =
\phi\big(\textstyle{\frac{|x|}{t}}\big)|\nabla {u_{\underline{0}}}(x)|^2\,|x_{n+1}|^a+ 
\dot{\phi}\big(\textstyle{\frac{|x|}{t}}\big){u_{\underline{0}}}(x)\nabla {u_{\underline{0}}}(x)\cdot\frac{x}{t\,|x|}\,|x_{n+1}|^a\\
&+\phi\big(\textstyle{\frac{|x|}{t}}\big){u_{\underline{0}}}(x)L_a({u_{\underline{0}}}(x)).
\end{align*}
Therefore, \eqref{e:D} follows from the divergence theorem by taking into 
account that $V$ is compactly supported.

Next, \eqref{e:Hprime} is a consequence of \eqref{e:D} and the direct computation
\begin{align*}
\frac{\d}{\d t}\big(H_{u_{\underline{0}}}(t)\big) 
&= \frac{\d}{\d t} \left(-t^{n+a} \int \dot{\phi}(|y|)\, 
\frac{u_{\underline{0}}^2(t\,y)}{|y|}\,|y_{n+1}|^a \,\d y\right)\\
& = \frac{n+a}{t}\,H_{u_{\underline{0}}}(t) - 2\,t^{n+a} \int \dot{\phi}(|y|)\, 
{u_{\underline{0}}}(t\,y)\,\nabla {u_{\underline{0}}}(t\,y)\cdot \frac{y}{|y|}\,|y_{n+1}|^a \,\d y\\
& =\frac{n+a}{t}\,H_{u_{\underline{0}}}(t)+2\,G_{u_{\underline{0}}}(t).
\end{align*}
Finally, to prove \eqref{e:Dprime} we consider the compactly supported vector field
$W \in C^\infty(B_1\setminus B_1',\R^{n+1})$ defined by
\[
W(x) = \left(\frac{|\nabla {u_{\underline{0}}}(x)|^2}{2}x-(\nabla {u_{\underline{0}}}(x)\cdot x)\nabla {u_{\underline{0}}}(x)\right)
\phi\big(\textstyle{\frac{|x|}{t}}\big)|x_{n+1}|^a\,.
\]
Moreover, conditions \eqref{e:bd condition1}-\eqref{e:bd distribution} and Lemma~\ref{l:grufolo polinomio} imply that 
$\lim_{y\downarrow (x',0)}W(y) \cdot e_{n+1} = 0$. 
Thus, $\div\,W$ has no singular part in $B_1'$, and we can compute pointwise the distributional divergence
as follows for $x_{n+1}\neq 0$
\begin{align*}
\div\, W (x)  = & \dot{\phi}\big(\textstyle{\frac{|x|}{t}}\big)
\,\frac{x}{t\,|x|} \cdot\Big(\frac{|\nabla {u_{\underline{0}}}(x)|^2}{2}x-(\nabla {u_{\underline{0}}}(x)\cdot x)\nabla {u_{\underline{0}}}(x)\Big)
|x_{n+1}|^a\\ 
&+ \phi\big(\textstyle{\frac{|x|}{t}}\big)\,\frac{n+a-1}{2}\,|\nabla {u_{\underline{0}}}(x)|^2|x_{n+1}|^a
- \phi\big(\textstyle{\frac{|x|}{t}}\big)\,(\nabla {u_{\underline{0}}}(x)\cdot x)L_a({u_{\underline{0}}}(x)).
\end{align*}
Therefore, we infer that
\begin{align*}
0= \int\div\,W(x)\,\d x  =&\int \dot{\phi}\Big(\textstyle{\frac{|x|}{t}}\Big)
\frac{|x|}{2\,t}\,|\nabla {u_{\underline{0}}}(x)|^2|x_{n+1}|^a\,\d x +
t\,E_{u_{\underline{0}}}(t) + \frac{n+a-1}{2}\,D_{u_{\underline{0}}}(t)\\
&-\int \phi\big(\textstyle{\frac{|x|}{t}}\big)\big(\nabla u_{\underline{0}}(x)\cdot x\big)
 L_a\big(u_{\underline{0}}(x)\big)\d x,
\end{align*}
and we conclude \eqref{e:Dprime} by direct differentiation since 
\[
\frac{\d}{\d t}\big(D_{u_{\underline{0}}}(t)\big)=-\int \dot{\phi}\Big(\textstyle{\frac{|x|}{t}}\Big)
\frac{|x|}{t^2}\,|\nabla {u_{\underline{0}}}(x)|^2|x_{n+1}|^a\,\d x.\qedhere
\]
\end{proof}
As a consequence we derive a first monotonicity formula for $H_{u_{x_0}}$ in $B_1$.
\begin{corollary}\label{c:H}
Let $u$ be a solution to the fractional obstacle problem \eqref{e:ob-pb local}. 
Then, for all $x_0 \in B'_1$ and $0 < r_0 < r_1<1-|x_0|$ such that 
$H_{u_{x_0}}(t)>0$ for all $t\in(r_0,r_1)$, 
we have
\begin{equation}\label{e:monotonia H}
\frac{H_{u_{x_0}}(r_1)}{r_1^{n+a}} = \frac{H_{u_{x_0}}(r_0)}{r_0^{n+a}}\,
e^{2\int_{r_0}^{r_1} \frac{I_{u_{x_0}}(t)}{t} \d t}.
\end{equation}
In particular, if $A_1 \leq I_{u_{x_0}}(t) \leq A_2$ 
for every $t \in (r_0, r_1)$, then 
\begin{gather}
(r_0, r_1)\ni r\mapsto\frac{H_{u_{x_0}}(r)}{r^{n+a + 2A_2}}
\quad\text{is monotone decreasing},\label{e:H1}\\
(r_0, r_1)\ni r\mapsto\frac{H_{u_{x_0}}(r)}{r^{n+a + 2A_1}}
\quad\text{is monotone increasing}.\label{e:H2}
\end{gather}
 Moreover, for all $x_0 \in B'_1$ and $0 < r <1-|x_0|$
 \begin{align}\label{e:L2 vs H}
 \int_{B_r(x_0)}|u_{x_0}|^2|x_{n+1}|^a\,\d x\leq r\,H_{u_{x_0}}(r).
 \end{align}
\end{corollary}
\begin{proof}
The proof of \eqref{e:monotonia H}
(and hence of \eqref{e:H1} and \eqref{e:H2})
follows from the differential equation in
\eqref{e:Hprime}. 
 The proof of \eqref{e:L2 vs H} is a simple consequence of a dyadic integration argument:
 \begin{align*}
 \int_{B_r(x_0)}|u_{x_0}|^2|x_{n+1}|^a\,\d x
 &= \sum_{j\in\N} \int_{B_{\sfrac{r}{2^j}}(x_0)
 \setminus B_{\sfrac{r}{2^{j+1}}}(x_0)}|u_{x_0}|^2|x_{n+1}|^a\,\d x\\
 &\leq \sum_{j\in\N} \frac{r}{2^j}\,H_{u_{x_0}}\big(\sfrac{r}{2^j}\big)
 \leq r\,H_{u_{x_0}}(r),
 \end{align*}
 where in the last inequality we used that $H_{u_{x_0}}(s)\leq H_{u_{x_0}}(r)$ for 
 $s\leq r$ by \eqref{e:H2} (with $A_1=0$).
\end{proof}

We establish next an auxiliary lemma containing useful bounds for some quantities 
related to the $L^2$-norm of $u_{x_0}$, for points $x_0$ in the contact set.
\begin{lemma}\label{l:estimates}
Let $u$ be a solution to the fractional obstacle problem \eqref{e:ob-pb local} in $B_1$.
Then, there is a positive constant $C_{\ref{l:estimates}}=C_{\ref{l:estimates}}(n,a)>0$ 
such that for every point $x_0 \in \Lambda_\varphi(u)$ we have for all $r\in(0,1-|x_0|)$
\begin{equation}\label{e:HD}
H_{u_{x_0}}(r)\leq C_{\ref{l:estimates}}
 \,\big(r\,D_{u_{x_0}}(r)+
 r^{n+a+2(k+1)}\big),
\end{equation}
 \begin{equation}\label{e:norma L2 D}
\int \phi\big({\textstyle{\frac{|x-x_0|}{r}}}\big)\,
|{u_{x_0}}(x)|^2\,|x_{n+1}|^a\d x \leq C_{\ref{l:estimates}}
\,\big(r^2\,D_{u_{x_0}}(r)+
r^{n+a+1+2(k+1)}\big),
\end{equation}
and 
\begin{equation}\label{e:norma L2 D bis}
\int_{B_r(x_0)\setminus B_{\frac r2}(x_0)}\,
|{u_{x_0}}(x)|^2\,|x_{n+1}|^a\d x \leq C_{\ref{l:estimates}}
\,\big(r^2\,D_{u_{x_0}}(r)+
r^{n+a+1+2(k+1)}\big).
\end{equation}
\end{lemma}
\begin{proof}
By the co-area formula for Lipschitz functions we check that 
 \begin{align}
 H_{u_{x_0}}(r) = 2\int_{\frac{r}{2}}^r \frac{\d t}{t}
 \int_{\de B_t(x_0)} |u_{x_0}(x)|^2\,|x_{n+1}|^a\,\d \cH^{n}(x),
 \label{e:formule integrate0}
 \end{align}
 and
\begin{align*}
D_{u_{x_0}}(r) & = 
\int_{B_{\frac r2}(x_0)}|\nabla u_{x_0}(x)|^2|x_{n+1}|^a\,\d x\\
&\quad +\frac 2r\int_{\frac r2}^r \d t
\int_{\partial B_t(x_0)}(r-t)|\nabla u_{x_0}(x)|^2\,|x_{n+1}|^a\,\d \cH^{n}(x).
\end{align*}
Therefore, an integration by parts gives
\begin{align}\label{e:formule integrate1}
D_{u_{x_0}}(r) = \frac{2}{r} \int_{\frac{r}{2}}^r \d t\int_{B_t(x_0)} 
|\nabla u_{x_0}(x)|^2\,|x_{n+1}|^a\,\d x.
\end{align}
By \eqref{e:Lav}, as $x_0\in\Lambda_\varphi(u)$, 
\cite[Lemma~2.13]{CaSaSi08} and \cite[Lemma~6.3]{GaRos17}
yield the Poincar\'e inequality 
\begin{equation}\label{e:CSS}
\frac 1t\int_{\de B_t(x_0)} |u_{x_0}(x)|^2\,|x_{n+1}|^a\,\d \cH^{n}(x)
\leq C\, \int_{B_t(x_0)} |\nabla u_{x_0}(x)|^2\,
|x_{n+1}|^a\,\d x+C
t^{n+a-1+2(k+1)},
\end{equation}
with $C=C(n,a)>0$. Integrating the latter inequality on $(\sfrac r2,r)$ we find \eqref{e:HD}
in view of \eqref{e:formule integrate1}. 
Instead, by first multiplying formula \eqref{e:CSS} by $t$ and then integrating 
over $(0,r)$, we infer
\[
\int_{B_r(x_0)} |u_{x_0}(x)|^2\,|x_{n+1}|^a\,\d x
\leq C\,r^2 \int_{B_r(x_0)} |\nabla u_{x_0}(x)|^2\,
|x_{n+1}|^a\,\d x+C
r^{n+a+1+2(k+1)}.
\]
In conclusion, \eqref{e:norma L2 D} and \eqref{e:norma L2 D bis} follow directly.
\end{proof}

Next we show an explicit expression for the radial derivative 
of $I_{u_{x_0}}$ at all points $x_0\in B_1'$. We follow here \cite[Proposition~2.7]{FoSp17}.
\begin{proposition}\label{p:monotonia+lower}
Let $u$ be a solution to the fractional obstacle problem \eqref{e:ob-pb local} in $B_1$.
Then, if $x_0 \in B_1'$ is such that $H_{u_{x_0}}(t)>0$
for all $t\in[r_0,r_1]$, we have
\begin{equation}\label{e:monotonia freq}
I_{u_{x_0}}(r_1) - I_{u_{x_0}}(r_0)  =
\int_{r_0}^{r_1}\Big(\frac{2\,t}{H_{u_{x_0}}^2(t)} \big(H_{u_{x_0}}(t) \; E_{u_{x_0}}(t) 
- G_{u_{x_0}}^2(t)\big) +R_{u_{x_0}}(t)\Big)\,\d t
\end{equation} 
for $0 < r_0 < r_1<1-|x_0|$, with
\begin{equation}\label{e:Ru0}
 |R_{u_{x_0}}(t)|\leq C_{\ref{p:monotonia+lower}}
\frac{t^{n+a+2k+1}}{H_{u_{x_0}}(t)}
\Big(\Big(\frac{D_{u_{x_0}}(t)}{t^{n+a+2k+1}}\Big)^{\sfrac12}+1\Big)
\end{equation}
and $C_{\ref{p:monotonia+lower}}=C_{\ref{p:monotonia+lower}}(n,a)>0$.
\end{proposition}
\begin{proof}
It is not restrictive to assume $x_0= \underline{0}$.
We use the identities in \eqref{e:D}, \eqref{e:Hprime} and \eqref{e:Dprime} to compute
(the lenghty details are left to the reader)
\begin{align*}
\frac{\d}{\d t}\big(I_{u_{\underline{0}}}(t)\big)&= {I_{u_{\underline{0}}}(t)}
\Big(\frac 1t+\frac{\frac{\d}{\d t}\big(G_{u_{\underline{0}}}(t)\big)}
{G_{u_{\underline{0}}}(t)}-\frac{\frac{\d}{\d t}\big(H_{u_{\underline{0}}}(t)\big)}{H_{u_{\underline{0}}}(t)}\Big)\\
& =2{I_{u_{\underline{0}}}(t)}\Big(\frac{E_{u_{\underline{0}}}(t)}{G_{u_{\underline{0}}}(t)}-\frac{G_{u_{\underline{0}}}(t)}{H_{u_{\underline{0}}}(t)}\Big)
+R_{u_{\underline{0}}}(t),
\end{align*} 
where  
\begin{align*} 
R_{u_{\underline{0}}}(t)&:=-\frac{1}{H_{u_{\underline{0}}}(t)}\int \phi\big(\textstyle{\frac{|x|}{t}}\big)
 \Big((n+a-1){u_{\underline{0}}}(x)+2(\nabla {u_{\underline{0}}}(x)\cdot x)\Big)L_a({u_{\underline{0}}}(x)) \d x\\
& +\frac1t\int\dot{\phi}\Big(\textstyle{\frac{|x|}{t}}\Big)|x|u(x)L_a({u_{\underline{0}}}(x)) \d x\,.
\end{align*}
From this we conclude \eqref{e:monotonia freq} straightforwardly. 

For \eqref{e:Ru0}, we estimate separately each term appearing in the integral 
defining $R_{u_{\underline{0}}}(t)$. We start with
\begin{align}\label{e:vLav}
  \left|\int \phi\big(\textstyle{\frac{|x|}{t}}\big)
 {u_{\underline{0}}}(x)L_a({u_{\underline{0}}}(x))\d x\right| & \stackrel{\eqref{e:Lav}}{\leq} 
 \left|\int \phi\big(\textstyle{\frac{|x|}{t}}\big)
 |{u_{\underline{0}}}(x)||x'|^{k-1}|x_{n+1}|^a\d x\right|\notag\\
 & \leq 
 t^{k-1}\left(\int_{B_t} |x_{n+1}|^a\d x\right)^{\sfrac12}
 \left(\int \phi\big(\textstyle{\frac{|x|}{t}}\big)|{u_{\underline{0}}}(x)|^2 |x_{n+1}|^a\d x\right)^{\sfrac12}
 \notag\\
 & \leq C
 t^{\frac{n+a+1}2+ k-1}
 \left(\int \phi\big(\textstyle{\frac{|x|}{t}}\big)|{u_{\underline{0}}}(x)|^2 |x_{n+1}|^a\d x\right)^{\sfrac12}
 \notag\\
 &\stackrel{\eqref{e:norma L2 D}}{\leq} C
t^{\frac{n+a-1}2+k+1}\Big(D^{\sfrac 12}_{u_{\underline{0}}}(t)+t^{\frac{n+a-1}2+k+1}\Big),
\end{align}
with $C_{\ref{e:vLav}}=C_{\ref{e:vLav}}(n,a)>0$. 
Arguing similarly we infer
\begin{align}\label{e:vLav3}
&\left| \int \phi\big(\textstyle{\frac{|x|}{t}}\big)
(\nabla {u_{\underline{0}}}(x)\cdot x)L_a({u_{\underline{0}}}(x))\d x\right| \leq   
t \int \phi\big(\textstyle{\frac{|x|}{t}}\big)|\nabla {u_{\underline{0}}}(x)||L_a({u_{\underline{0}}}(x))|\d x\notag\\ 
& \stackrel{\eqref{e:Lav}}{\leq} t^k \int \phi\big(\textstyle{\frac{|x|}{t}}\big)|\nabla {u_{\underline{0}}}(x)||x_{n+1}|^a\d x\leq
C
t^{\frac{n+a-1}2+ k+1}D_{u_{\underline{0}}}^{\sfrac 12}(t),
\end{align}
and
\begin{align}\label{e:vLav2}
\left|\frac1t\int\dot{\phi}\Big(\textstyle{\frac{|x|}{t}}\Big)|x|{u_{\underline{0}}}(x)L_a({u_{\underline{0}}}(x)) \d x\right|
& \stackrel{\eqref{e:Lav}}{\leq} 
t^{k-1}\int_{B_t\setminus B_{\sfrac t2}}|{u_{\underline{0}}}(x)| |x_{n+1}|^a \d x\notag\\
& \stackrel{\eqref{e:norma L2 D bis}}{\leq} C\,
t^{\frac{n+a-1}2+k+1}\Big( D^{\sfrac 12}_{u_{\underline{0}}}(t)+t^{\frac{n+a-1}2+k+1}\Big),
\end{align}
with $C=C(n,a)>0$. 
Therefore, \eqref{e:Ru0} follows at once from \eqref{e:vLav}-\eqref{e:vLav2}.
\end{proof}

Estimate \eqref{e:Ru0} turns out to be useful to analyze the subsets of points 
$\Gamma_{\varphi,\theta}(u)$ of $\Gamma_\varphi(u)$, for every $\theta\in(0,1)$ (cf. \eqref{e:Hgrowth intro}). 
With fixed $\theta\in (0,1)$, we then look at points of the free boundary in the subset
\begin{equation}\label{e:Hgrowth}
{\mathscr{Z}_{\varphi,\theta,\delta}}(u):=\big\{x_0\in \Gamma_\varphi(u)\cap B_{\sfrac12}:\, 
H_{u_{x_0}}(r)\geq\delta\, r^{n+a+2(k+1-\theta)}\quad\forall\; r\in(0,\sfrac12)\big\},
\end{equation}
where $\delta>0$ is any. 

\begin{remark}
Note that ${\mathscr{Z}_{\varphi,\theta,\delta}}(u)\subseteq {\mathscr{Z}_{\varphi,\theta,\delta'}}(u)$ 
if $\delta'\leq \delta$. Hence, in what follows it is enough to consider the values of $\delta$ small enough.
\end{remark}

\begin{proposition}\label{p:almost monotonicity}
For every $\delta>0$, there exist $C_{\ref{p:almost monotonicity}},\varrho_{\ref{p:almost monotonicity}}>0$
such that for every $x_0 \in \ZZ(u)$, the function 
$(0,\varrho_{\ref{p:almost monotonicity}}]\ni r\mapsto  e^{C_{\ref{p:almost monotonicity}}
r^{\theta}}I_{u_{x_0}}(r)$ is nondecreasing. 
In particular, the ensuing limits exist finite and are equal 
\begin{equation}\label{e:old new freq}
 \lim_{r\downarrow 0}\frac{rD_{u_{x_0}}(r)}{H_{u_{x_0}}(r)}
 =\lim_{r\downarrow 0} I_{u_{x_0}}(r)=:I_{u_{x_0}}(0^+).
\end{equation}
\end{proposition}
\begin{proof}
Since $x_0 \in \ZZ(u)$, formula \eqref{e:HD} yields for $r\in(0,\sfrac12)$
\[
C_{\ref{l:estimates}} D_{u_{x_0}}(r) \geq \delta
 r^{n+a-1+2(k+1-\theta)} -C_{\ref{l:estimates}}
 r^{n+a-1+2(k+1)}\,,
\]
therefore, for $\varrho_{\ref{p:almost monotonicity}}$ sufficiently small, we have for all $r\in(0,\varrho_{\ref{p:almost monotonicity}}]$
 \begin{equation}\label{e:Ddelta}
 D_{u_{x_0}}(r) \geq C {r^{n+a-1+2(k+1-\theta)}}.
\end{equation}
In addition, from \eqref{e:D} and \eqref{e:vLav} we get for all $r\in(0,\varrho_{\ref{p:almost monotonicity}}]$, 
if $\varrho_{\ref{p:almost monotonicity}}$ small enough,
\begin{align}\label{e:vLav1}
   |G_{u_{x_0}}(r) - D_{u_{x_0}}(r)|&\stackrel{\eqref{e:D}}{=}
   \left|\int \phi\big(\textstyle{\frac{|x-x_0|}{r}}\big)u_{x_0}(x)L_a(u_{x_0}(x))\d x\right|\notag\\ 
   & \stackrel{\eqref{e:vLav}}{\leq} C_{\ref{e:vLav}}
   D_{u_{x_0}}(r)\Big(
   \frac{r^{n+a+2k+1}}{D_{u_{x_0}}(r)}+\Big(\frac{r^{n+a+2k+1}}{D_{u_{x_0}}(r)}\Big)^{\sfrac 12}\Big)\notag\\ &
   \stackrel{\eqref{e:Ddelta}}{\leq} C
   D_{u_{x_0}}(r)\Big(2Cr^{2\theta}+\big(2Cr^{2\theta}\big)^{\sfrac12}\Big) 
  \leq C
  r^\theta D_{u_{x_0}}(r).
    \end{align}
Therefore, from \eqref{e:Ru0}, if $\varrho_{\ref{p:almost monotonicity}}$ is sufficiently small,
we get for all $r\in(0,\varrho_{\ref{p:almost monotonicity}}]$, 
\begin{align}\label{e:Ru}
 |R_{u_{x_0}}(r)|& \leq C_{\ref{p:monotonia+lower}}\,
 I_{u_{x_0}}(r)\frac{r^{n+a+2k}}{G_{u_{x_0}}(r)}
 \Big(\Big(\frac{D_{u_{x_0}}(r)}{r^{n+a+2k+1}}\Big)^{\sfrac 12}+1\Big)\notag\\
 &\stackrel{\eqref{e:vLav1
 }}{\leq} \,C\,
 \frac{I_{u_{x_0}}(r)}{r}
 \Big(\Big(\frac{r^{n+a+2k+1}}{D_{u_{x_0}}(r)}\Big)^{\sfrac 12}
 +\frac{r^{n+a+2k+1}}{D_{u_{x_0}}(r)}\Big)
 \stackrel{\eqref{e:Ddelta}}{\leq} C\,
 r^{\theta-1}\,I_{u_{x_0}}(r).
 \end{align}
Hence, from \eqref{e:Cauchy Schwarz}, \eqref{e:monotonia freq} and \eqref{e:Ru} we find
 \begin{equation}\label{e:derivata monotonia}
 \frac{\d}{\d r}\big( I_{u_{x_0}}(r)\big)\geq -C
 \,r^{\theta-1}I_{u_{x_0}}(r),
 \end{equation}
and the monotonicity of $(0,\varrho_{\ref{p:almost monotonicity}}]\ni 
r\mapsto e^{C_{\ref{p:almost monotonicity}}
r^{\theta}}I_{u_{x_0}}(r)$ follows by direct integration. 
In addition, we also infer \eqref{e:old new freq}, because
from \eqref{e:vLav1} 
for all $r\in(0,\varrho_{\ref{p:almost monotonicity}}]$ we have
\begin{equation}\label{e:frequencies comparison}
(1-C
r^{\theta}) \frac{rD_{u_{x_0}}(r)}{H_{u_{x_0}}(r)}
\leq I_{u_{x_0}}(r)\leq (1+C
r^{\theta})\frac{rD_{u_{x_0}}(r)}{H_{u_{x_0}}(r)}.
\end{equation}
\end{proof}
\begin{remark}
 The monotonicity for the truncated Almgren's frequency function 
  \[
  r(1+Cr^\theta)\frac{\d}{\d r}\log\max\big\{H_{u_{x_0}}(r),r^{n+a+2(k+1-\theta)}\big\}
 \]
 proved in \cite{CaSaSi08} and \cite{GaRos17} is essentially equivalent to Proposition~\ref{p:almost monotonicity}.
\end{remark}

We derive next an additive quasi-monotonicity formula for the frequency.
\begin{corollary}\label{c:monotone additive}
For every $A,\delta>0$, there exist $C_{\ref{c:monotone additive}}$,  $\varrho_{\ref{c:monotone additive}}>0$ with this property:
if $x_0 \in \ZZ(u)$ and  $I_{u_{x_0}}(\varrho_{\ref{c:monotone additive}}) \leq A$, then 
for all $\Lambda\geq A\,C_{\ref{c:monotone additive}}$ the function
\begin{equation}\label{e:monotone additive}
(0,\varrho_{\ref{c:monotone additive}}]\ni r\mapsto I_{u_{x_0}}(r)
+\Lambda\,r^{\theta}\quad\text{is nondecreasing.}
\end{equation}
\end{corollary}
\begin{proof}
Under the standing assumptions, the quasi-monotonicity of
$I_{u_{x_0}}$ and \eqref{e:derivata monotonia} yield that
\[
 \frac{\d}{\d r}\big(I_{u_{x_0}}(r)\big)\geq 
 -Ce^{C_{\ref{p:almost monotonicity}}}\,A\, r^{\theta-1},
\]
for $r$ sufficiently small.
Hence, we conclude \eqref{e:monotone additive} at once by integration. 
\end{proof}

\subsection{Lower bound on the frequency and compactness}
We first show that the frequency of a solution $u$ to 
\eqref{e:ob-pb local} at points in $\ZZ(u)$ is bounded from below 
by a universal constant.

\begin{lemma}\label{l:freq lower}
For every $\delta>0$ there exists $\varrho_{\ref{l:freq lower}}>0$ such that, 
for all $x_0 \in \ZZ(u)$ and $r\in(0,\varrho_{\ref{l:freq lower}}]$, 
\begin{equation}\label{e:freq lower}
I_{{u_{x_0}}}(r) \geq \frac 1{2C_{\ref{l:estimates}}}. 
\end{equation}
\end{lemma}
\begin{proof}
 In view of \eqref{e:HD} and since $x_0 \in \ZZ(u)$, we have for all $r$ sufficiently small,
 \begin{equation*}
\frac{1}{C_{\ref{l:estimates}}}\leq
\frac{rD_{u_{x_0}}(r)}{H_{u_{x_0}}(r)}+
\frac{r^{2\theta}}{\delta}.
 \end{equation*}
 Inequality \eqref{e:freq lower} is a straightforward consequence of estimate \eqref{e:frequencies comparison} and
the latter estimate provided that $\varrho_{\ref{l:freq lower}}$ is sufficiently small. 
\end{proof}

For the free boundary analysis developed in \cite{FoSp17} it is mandatory to consider the critical set 
of a solution. In the current framework, the natural subsitute for the critical set is given by
\[
\mathscr{N}_\varphi(u) := \Big\{
(x',0)\in B_1'\,:\,u(x',0)-\varphi(x')=|\nabla' \big(u(x',0)-\varphi(x')\big)|
= \lim_{y\downarrow 0^+}t^a\de_{n+1}u(x',y) = 0
\Big\}.
\]
Notice that $\Gamma_\varphi(u)\subseteq\mathscr{N}_\varphi(u)\subseteq\Lambda_\varphi(u)$
(the first inclusion is a consequence of \eqref{e:bd condition1}). 

We can then give the following compactness result. 
For $u:B_1\to\R$ solution of \eqref{e:ob-pb local} and $x_0\in B_1'$ we introduce the rescalings
\begin{equation}\label{e:rescaling0}
u_{x_0,r} (y) := 
\frac{r^{\frac{n+a}{2}}\,u_{x_0}(x_0+ry)}{H_{u_{x_0}}^{\sfrac12}(r)}
\qquad \forall \; r\in(0,1-|x_0|), \;\forall\; y \in B_{\frac{1-|x_0|}{r}}.
\end{equation}
Note that $u_{x_0,r}$ is a minimizer of  the functional
\begin{align}\label{e:funz rescaling2}
\int_{B_1}|\nabla v|^2|x_{n+1}|^a\d x
-2\int_{B_1}v L_a(\varphi_{x_0,r})\,\d x
\end{align}
with obstacle function
\begin{equation}\label{e:rescaling phi}
\varphi_{x_0,r} (y) := 
\frac{r^{\frac{n+a}{2}}\,\varphi_{x_0}(x_0+ry)}{H_{u_{x_0}}^{\sfrac12}(r)}\,,
\end{equation}
among all functions $v\in 
u_{x_0,r}+H^1_0(B_1,\dm)$ satisfying 
$v(x',0)\geq 0$ on $B_1'$.
\begin{corollary}\label{c:compactness}
Let $\delta>0$ be given. Let $(u_l)_{l\in\N}$ be a sequence of solutions to the fractional obstacle problem 
\eqref{e:ob-pb local} in $B_1$ with obstacle functions $\varphi_l$
equi-bounded in $C^{k+1}(B_1)$, and let
$x_l\in \mathscr{Z}_{\varphi_l,\theta,\delta}(u_l)$ 
be such that $\sup_{l} I_{(u_l)_{x_l}}(\varrho_l) <+\infty$, for some $\varrho_l\downarrow 0$.

Then, there exist a subsequence $l_j\uparrow\infty$ and a solution $v_\infty$ to the fractional obstacle problem 
\eqref{e:ob-pb local} in $B_1$ with null obstacle function, such that on setting 
$v_j:=(u_{l_j})_{x_{l_j},\varrho_{l_j}}$ we have
\begin{gather}
v_j\to v_\infty \quad \text{in $H^1(B_1, \dm)$,}\label{e:cpt1}\\
v_j \to v_\infty \quad \text{in $C^{0,\alpha}_{\loc}(B_1)$, 
$\forall\;\alpha <\min\{1,2s\}$} 
\label{e:cpt4}\\
\nabla' v_j \to \nabla' v_\infty \quad 
\text{in $C^{0,\alpha}_{\loc}(B_1)$, $\forall\;\alpha <s$}, 
\label{e:cpt2}\\
\sgn(x_{n+1})\,|x_{n+1}|^a\de_{x_{n+1}} v_j \to 
\sgn(x_{n+1})\,|x_{n+1}|^a\de_{x_{n+1}} v_\infty 
\text{ in $C^{0,\alpha}_{\loc}(B_1)$, $\forall\;\alpha <1-s$.}
\label{e:cpt3}
\end{gather}
\end{corollary}
\begin{proof}
 By taking into account inequality \eqref{e:frequencies comparison} in 
 Proposition~\ref{p:almost monotonicity} we get for $l$ large
 \[
 \frac{\varrho_l D_{(u_l)_{x_l}}(\varrho_l)}{H_{(u_l)_{x_l}}(\varrho_l)}
 \leq (1+C\|\varphi_{l_j}\|_{C^{k+1}(B_1')}\varrho_l^\theta) I_{(u_l)_{x_l}}(\varrho_l).
 \]
In particular, we infer that $\sup_lD_{(u_l)_{x_l,\varrho_l}}(1)<\infty$. 
Thus, a subsequence $v_j:=(u_{l_j})_{x_{l_j},\varrho_{l_j}}$ converges weakly
$H^1(B_1,\dm)$ to some function $v_\infty$. Moreover, $v_j$ is a local minimizer of 
 \begin{align*}
F_j(v):=\int_{B_1}|\nabla v|^2|y_{n+1}|^a\d y
-2\int_{B_1}v L_a\big((\varphi_{l_j})_{x_{l_j},\varrho_{l_j}}\big)\,\d y
\end{align*}
among all functions $v\in v_j+H^1_0(B_1,\dm)$ satisfying $v(x',0)\geq 0$ on $B_1'$ 
(cf. \eqref{e:rescaling0}-\eqref{e:funz rescaling2}). 

By taking into account that $x_{l_j}\in 
\mathscr{Z}_{\varphi_{l_j},\theta, \delta}(u_{l_j})$, 
inequality \eqref{e:Laphi} implies that for all $y\in B_1\setminus B_1'$
\begin{equation}\label{e:Laphij}
|\big(L_a(\varphi_{l_j})_{x_{l_j},\varrho_{l_j}}\big)(y)|\leq 
\frac 1{\delta^{\sfrac12}}\|\varphi_{l_j}\|_{C^{k+1}(B_1')}
\varrho_{l_j}^{\theta}
|y_{n+1}|^a.
\end{equation}
Therefore, one can easily show that the sequence $(F_j)_j$ 
$\Gamma(L^2(B_1,\dm))$-converges to 
the functional $F_\infty:L^2(B_1,\dm)\to[0,+\infty]$ defined by
\[
F_\infty(v):=\int_{B_1}|\nabla v|^2|y_{n+1}|^a\d y
\]
if $v\in v_\infty+H^1_0(B_1,\dm)$ with $v(x',0)\geq 0$ on $B_1'$, and $+\infty$ otherwise on $L^2(B_1,\dm)$. 
In addition, being the $F_j$'s equicoercive in $L^2(B_1,\dm)$,
$F_j(v_j)\to F_\infty(v_\infty)$, so that by \eqref{e:Laphij} the convergence of $(v_j)_j$ to $v_\infty$ is actually strong 
$H^1(B_1,\dm)$.

Items \eqref{e:cpt4}-\eqref{e:cpt3} are then a straightforward consequence 
of Theorem~\ref{t:reg} and \eqref{e:Laphij} (cf. the arguments in 
\cite[Lemma~6.2]{CaSaSi08}).
 \end{proof}

A sharp lower bound on the frequency then follows.
\begin{corollary}\label{c:minima freq}
Let $\delta>0$. If $x_0 \in \ZZ(u)$, then
\begin{equation}\label{e:lower bound}
I_{u_{x_0}}(0^+) \geq 1+s\;. 
\end{equation}
\end{corollary}
\begin{proof}
 Note that $I_{u_{x_0}}(0^+)=\lim_{r\downarrow 0}I_{u_{x_0}}(r)=
 \lim_{r\downarrow 0}I_{u_{x_0,r}}(1)= I_{v_\infty}(1)$, for some $v_\infty$ 
 homogeneous solution to the fractional obstacle problem \eqref{e:ob-pb local} 
 with null obstacle function provided by 
Corollary~\ref{c:compactness}. 
 Thus, we conclude \eqref{e:lower bound} by \cite[Proposition~5.1]{CaSaSi08} 
 (see also \cite[Corollary~2.12]{FoSp17}).
\end{proof}

\section{Main estimates on the frequency}\label{s:frequency estimate}

In this section we prove the principal estimates on the frequency
that we are going to exploit in the sequel. We start with an elementary lemma. Recall that all obstacles functions $\varphi$ 
are assumed to satisfy the normalization condition
$\normphi_{C^{k+1}(B_1')}\leq 1$.
\begin{lemma}\label{l:lim uniforme}
Let $A,\,\delta>0$.  Then, there exist
$C_{\ref{l:lim uniforme}}$, $\varrho_{\ref{l:lim uniforme}}>0$
such that, if $u$ is a solution of to the fractional obstacle problem 
\eqref{e:ob-pb local} in $B_1$, with $\underline{0}\in \ZZ(u)$ and
$I_{u_{\underline{0}}}(2\varrho)\leq A$, $\varrho\leq \varrho_{\ref{l:lim uniforme}}$,
then for every $x \in B_{\sfrac\varrho2}'$
\begin{gather}
\label{e:H limitato}
\frac{1}{C_{\ref{l:lim uniforme}}} \leq 
\frac{H_{u_x}(\varrho)}{H_{u_{\underline{0}}}(\varrho)} 
\leq C_{\ref{l:lim uniforme}}\and 
\frac{1}{C_{\ref{l:lim uniforme}}}\leq 
\frac{D_{u_x}(\varrho)}{D_{u_{\underline{0}}}(\varrho)} 
\leq C_{\ref{l:lim uniforme}},\\
\Big\vert I_{u_{\underline{0}}}(\varrho) - I_{u_x}(\varrho) 
\Big\vert\leq C_{\ref{l:lim uniforme}}\,.\label{e:I limitato}
\end{gather}
\end{lemma}
\begin{remark}\label{r:freq ben def}
Note that as a byproduct of the first estimate in \eqref{e:H limitato} 
in Lemma~\ref{l:lim uniforme} frequencies at the scale $\varrho$ 
are well-defined at every point $x \in B_{\sfrac\varrho2}'$, 
recalling that $\underline{0}\in \ZZ(u)$.
\end{remark}
\begin{proof}
In order to prove \eqref{e:H limitato}, we argue by contradiction: 
we can assume that there exist $A,\,\delta>0$ and solutions $u_j$
to the fractional obstacle problem with obstacles $\varphi_j$, $\|\varphi_j\|_{C^{k+1}(B_1')}\leq1$, 
with $\underline{0}\in \mathscr{Z}_{\varphi_j,\theta,\delta}(u_j)$, such that
$I_{(u_j)_{\underline{0}}}(\varrho_j) \leq A$, for some $\varrho_j\downarrow0$, and there exist points 
$x_j\in B_{\sfrac{\varrho_{j}}4}'$ contradicting one of the sets of inequalities in \eqref{e:H limitato}.

In particular, by almost monotonicity of the frequency function (cf. Proposition~\ref{p:almost monotonicity}) 
and the lower bound on the frequency (cf. Corollary~\ref{c:minima freq}) we infer that 
$1+s\leq I_{(u_j)_{\underline{0}}}(t)\leq A\,e^{C_{\ref{p:almost monotonicity}}(2\varrho_j)^{\theta}}
\leq A\,e^{C_{\ref{p:almost monotonicity}}}=:A'$ 
for all $t\in(0,2\varrho_j]$. By Corollary \ref{c:compactness}, up to a subsequence, $v_j:=(u_j)_{\underline{0},\varrho_j}$ 
converges strongly in $H^1(B_2,\dm)$ to a function $v_\infty$ solution of the fractional 
obstacle problem in $B_2$ with zero obstacle function. We assume in addition that 
$\varrho_j^{-1}x_j\to x_\infty\in\bar{B}_{\sfrac12}'$. 

To prove the first set of inequalities in \eqref{e:H limitato}, we compute
\begin{align}\label{e:due maroni0}
\frac{H_{(u_j)_{x_j}}(\varrho_j)}{H_{(u_j)_{\underline{0}}}(\varrho_j)}&=\frac{2\varrho_j^{n+a}}{H_{(u_j)_{\underline{0}}}(\varrho_j)}
 \int_{B_1\setminus B_{\sfrac 12}}(u_j)_{x_j}^2(x_j+\varrho_j x)\frac{|x_{n+1}|^a}{|x|}\d x\notag\\
=& \frac{2\varrho^{n+a}_j}{H_{(u_j)_{\underline{0}}}(\varrho_j)}\int_{B_1\setminus B_{\sfrac 12}}
\Big[u_j(x_j+\varrho_j x)-(\varphi_j)_{x_j}(x_j+\varrho_j {x'})\Big]^2\frac{|x_{n+1}|^a}{|x|}\d x\notag\\
=& 2\int_{B_1\setminus B_{\sfrac 12}}\Big[
v_j(\varrho_j^{-1}x_j+x)+\frac{\varrho_j^{\sfrac{(n+a)}2}}{H_{(u_j)_{\underline{0}}}^{\sfrac12}(\varrho)}
\big((\varphi_j)_{\underline{0}}(x_j+\varrho_j{x'})-(\varphi_j)_{x_j}(x_j+\varrho_j{x'})\big)\Big]^2\frac{|x_{n+1}|^a}{|x|}\d x\,.
\end{align}
Moreover, by estimate \eqref{e:stima rescal} in Remark~\ref{r:rescal} and since $\varrho_j^{-1}x_j\to x_\infty$ 
we get for all $x'\in B'_1$
\begin{align}\label{e:differenza polinomi}
|(\varphi_j)_{\underline{0}}(x_j+\varrho_j {x'})- (\varphi_j)_{x_j}(x_j+\varrho_j {x'})|
\leq C \varrho_j^{k+1}\,.
\end{align}
Therefore, recalling that $\underline{0}\in \mathscr{Z}_{\varphi_j,\theta,\delta}(u_j)$,
from \eqref{e:differenza polinomi} we infer
\begin{align}\label{e:due maroni1}
 \frac{\varrho_j^{n+a}}{H_{(u_j)_{\underline{0}}}(\varrho_j)}
 \int_{B_1\setminus B_{\frac 12}}
 \big((\varphi_j)_{\underline{0}}(x_j+\varrho_j {x'})-(\varphi_j)_{x_j}(x_j+\varrho_j {x'})\big)^2\frac{|x_{n+1}|^a}{|x|}\d x
 \leq \frac{C 
 }{\delta}\varrho_j^{2\theta}.
\end{align}
Since $\varrho_j\downarrow0$, by contradiction 
$\lim_j\frac{H_{(u_j)_{x_j}}(\varrho_j)}{H_{(u_j)_{\underline{0}}}(\varrho_j)}\in\{0,\infty\}$.
Moreover, by \eqref{e:due maroni0} and \eqref{e:due maroni1}, by the strong $L^2(B_1,\dm)$ and local uniform convergence of 
$v_j\to v_\infty$ we conclude that, 
\begin{align*}
&2\int_{B_1(x_\infty)\setminus B_{\sfrac12}(x_\infty)}
 v_\infty^2(y)\frac{|y_{n+1}|^a}{|y-x_\infty|}\d y=
2\lim_j \int_{B_1(\varrho_j^{-1}x_j)\setminus B_{\sfrac12}(\varrho_j^{-1}x_j)}
v_j^2(y)\frac{|y_{n+1}|^a}{|y-\varrho_j^{-1}x_j|}\d y\notag\\
&=2\lim_j\int_{B_1\setminus B_{\sfrac 12}}v_j^2(\varrho_j^{-1}x_j+x)\frac{|x_{n+1}|^a}{|x|}\d x=
\lim_j \frac{H_{(u_j)_{x_j}}(\varrho_j)}{H_{(u_j)_{\underline{0}}}(\varrho_j)}\,.
\end{align*}
Being the left hand side finite, necessarily 
\[
2\int_{B_1(x_\infty)\setminus B_{\sfrac12}(x_\infty)}v_\infty^2(y)\frac{|y_{n+1}|^a}{|y-x_\infty|}\d y=
 \lim_j\frac{H_{(u_j)_{x_j}}(\varrho_j)}{H_{(u_j)_{\underline{0}}}(\varrho_j)}=0\,.
\]
Hence, $v_\infty\equiv 0$ on $B_1(x_\infty)\setminus B_{\sfrac12}(x_\infty)$, and thus
$v_\infty\equiv 0$ on the whole of $B_1$ by analiticity. 
A contradiction to $H_{v_\infty}(1)=1$ that follows from strong $L^2(B_1,\dm)$ convergence and 
the equality $H_{v_j}(1)=1$ for all $j$.

The second set of inequalities in \eqref{e:H limitato} is proven by the same argument.
Indeed, assuming that $\lim_j\frac{D_{(u_j)_{x_j}}(\varrho_j)}{D_{(u_j)_{\underline{0}}}(\varrho_j)}\in\{0,\infty\}$
we have
\begin{align*}
 & \frac{D_{(u_j)_{x_j}}(\varrho_j)}{D_{(u_j)_{\underline{0}}}(\varrho_j)}=\frac{\varrho_j^{n+a+1}}{D_{(u_j)_{\underline{0}}}(\varrho_j)}
 \int_{B_1}\phi(|x|)|\nabla\big(u_j(x_j+\varrho_j x)-(\varphi_j)_{x_j}(x_j+\varrho_j{x'})\big)|^2|x_{n+1}|^a\d x\,,
\end{align*}
and sinceby \eqref{e:stima grad rescal} in Remark~\ref{r:rescal} and by \eqref{e:Ddelta} 
\begin{align*}
 \frac{\varrho_j^{n+a+1}}{D_{(u_j)_{\underline{0}}}(\varrho_j)}&
 \int_{B_1}\phi(|x|)|\nabla\big((\varphi_j)_{\underline{0}}(x_j+\varrho_j{x'})-(\varphi_j)_{x_j}(x_j+\varrho_j{x'})\big)|^2
 |x_{n+1}|^a\d x\\ &
 \leq C
 \frac{\varrho_j^{n+a+1+2k}}{D_{(u_j)_{\underline{0}}}(\varrho_j)}\leq\frac {C
 }{\delta}\varrho_j^{2\theta}\,,
\end{align*}
we get (recall $\varrho_j\downarrow0$)
\begin{align*}
 \lim_j\frac{D_{(u_j)_{x_j}}(\varrho_j)}{D_{(u_j)_{\underline{0}}}(\varrho_j)}=
 \lim_j\frac{1}{4I_{(u_j)_{\underline{0}}}(\varrho_j)}
 \int_{B_1}\phi(|x|)|\nabla v_j(\varrho_j^{-1}x_j+x)|^2|x_{n+1}|^a\d x.
\end{align*}
By the strong convergence of $v_j$ to $v_\infty$ in $H^1(B_1,\dm)$, 
we infer that the left hand side is finite and then actually $0$, 
so that
\[ 
\int_{B_1}\phi(|x|)|\nabla v_\infty(x_\infty+x)|^2|x_{n+1}|^a\d x=0.
\]
Thus, by analiticity $v_\infty$ is constant on $B_1$, and we may conclude that 
\[ 
\int_{B_1}\phi(|x|)|\nabla v_\infty(x)|^2|x_{n+1}|^a\d x=0.
\]
The latter equality contradicts
\[ 
\int_{B_1}\phi(|x|)|\nabla v_\infty(x)|^2|x_{n+1}|^a\d x\in[1+s,2A']\,,
\]
that follows from strong $H^1(B_1,\dm)$ convergence and recalling that 
$H_{v_j}(1)=1$ and $1+s\leq I_{v_j}(1)\leq 2D_{v_j}(1)\leq A'$ for $j$ big enough (cf. \eqref{e:frequencies comparison}).

Finally, \eqref{e:I limitato} follows straightforwardly from \eqref{e:H limitato} for $\varrho_{\ref{l:lim uniforme}}$ 
sufficiently small by taking into account \eqref{e:frequencies comparison}:
\begin{align*}
\Big\vert I_{u_{\underline{0}}}(r) - I_{u_{x}}(r) \Big\vert
& = \Big\vert \frac{rG_{u_{\underline{0}}}(r)}{H_{u_{x}}(r)} 
\Big( \frac{H_{u_{x}}(r)}{H_{u_{\underline{0}}}(r)} - 
\frac{G_{u_{x}}(r)}{G_{u_{\underline{0}}}(r)}\Big)  
\Big\vert \stackrel{\eqref{e:H limitato}}{\leq} C. \qedhere
\end{align*}
\end{proof}

We introduce next the following notation for the radial variation of 
(modified) frequency at a point $x \in \ZZ(u)$ 
of a solution $u$ in $B_1$: given $0<r_0<r_1<1-|x|$, we set
\[
\Delta^{r_1}_{r_0}(x) := I_{u_x}(r_1)+\Lambda\, r_1^{\theta}
- \big(I_{u_x}(r_0)+\Lambda\,r_0^{\theta}\big).
\]
Note that,  $\Delta^{r_1}_{r_0}(x)\geq 0$ if $x\in\ZZ(u)$, if
$r_1$ is sufficiently small and if $\Lambda\geq A\,C_{\ref{c:monotone additive}}$ (cf. Corollary~\ref{c:monotone additive}).
We do not indicate the dependence of $\Delta^{r_1}_{r_0}$ on 
$\Lambda$ since such a parameter will be fixed appropriately in 
the next result.
\begin{lemma}\label{l:monotonia}
Let $A,\,\delta>0$. Then, there exist $C_{\ref{l:monotonia}}$ and 
$\varrho_{\ref{l:monotonia}}>0$ such that, if $x_0\in\ZZ(u)$, 
and $I_{u_{x_0}}(r_1) \leq A$, with $2r_1\leq \varrho_{\ref{l:monotonia}}$, 
then for every $r_0\in (\sfrac{r_1}{8}, r_1)$ we have
\begin{align}\label{e:monotonia con resto}
\int_{B_{\sfrac{r_1}2}(x_0)\setminus B_{\sfrac{r_0}2}(x_0)} 
\Big(\nabla u_{x_0}(z) \cdot (z - x_0) - 
I_{u_{x_0}}(\sfrac{r_0}2)\,u_{x_0}(z)\Big)^2 
\frac{|z_{n+1}|^a}{|z-x_0|}\,\d z
\leq C_{\ref{l:monotonia}} 
H_{u_{x_0}}(r_1)\,\Delta^{r_1}_{\sfrac{r_0}2}(x_0).
\end{align}
\end{lemma}

\begin{proof} Without loss of generality we prove the result for $x_0=\underline{0}$.
We start off with the following computation:
\begin{align}\label{e:I mon}
2 \int_{B_t\setminus B_{\sfrac{t}{2}}} &
\Big(\nabla {u_{\underline{0}}}(z) \cdot z - I_{u_{\underline{0}}}(t)\,{u_{\underline{0}}}(z)\Big)^2 
\frac{|z_{n+1}|^a}{|z|} \d z\notag\\
&= \int - \dot{\phi}\big(\textstyle{\frac{|z|}{t}}\big)
\Big(\nabla {u_{\underline{0}}}(z) \cdot z - I_{u_{\underline{0}}}(t)\,{u_{\underline{0}}}(z)\Big)^2 
\frac{|z_{n+1}|^a}{|z|} \d z
\notag\allowdisplaybreaks\\
&= t^2 E_{u_{\underline{0}}}(t) - 2\, t\,I_{u_{\underline{0}}}(t)\,G_{u_{\underline{0}}}(t) 
+ I_{u_{\underline{0}}}^2(t) H_{u_{\underline{0}}}(t)\notag
\allowdisplaybreaks\\
&= \frac{t^2}{H_{u_{\underline{0}}}(t)}\,\big(E_{u_{\underline{0}}}(t) H_{u_{\underline{0}}}(t) 
- G_{u_{\underline{0}}}^2(t)\big)\stackrel{\eqref{e:monotonia freq}}{=} 
\frac{t}{2}\,H_{u_{\underline{0}}}(t)\Big(\frac{\d}{\d t}\big(
I_{u_{\underline{0}}}(t)\big)-R_{u_{\underline{0}}}(t)\Big)\,.
\end{align}
We now use the following integral estimate (whose elementary proof
is left to the readers)
\begin{align}\label{e:fubini}
\int_{B_{\rho_1}\setminus B_{\rho_0}} f(z) \d z& \leq 
\rho_0^{-1}\int_{\rho_0}^{2\rho_1} \int_{B_t\setminus B_{\sfrac{t}{2}}}f(z)\, \d z\,\d t
\qquad\text{$\forall$ $0<\rho_0\leq \rho_1$},
\end{align}
$f\geq0$ a measurable function, in order to deduce
\begin{align}\label{e:prima stima0}
&\int_{B_{\sfrac{r_1}2}\setminus B_{\sfrac{r_0}2}}
\Big(\nabla {u_{\underline{0}}}(z) \cdot z - 
I_{u_{\underline{0}}}(\sfrac{r_0}2)\,{u_{\underline{0}}}(z)\Big)^2 
\frac{|z_{n+1}|^a}{|z|}\,\d z\notag
\allowdisplaybreaks\\ & 
\stackrel{\eqref{e:fubini}}{\leq}
\frac2{r_0}\int_{\sfrac{r_0}2}^{r_1} 
\int_{B_t\setminus B_{\sfrac{t}{2}}}
\Big(\nabla {u_{\underline{0}}}(z) \cdot z - 
I_{u_{\underline{0}}}(\sfrac{r_0}2)\,{u_{x_0}}(z)\Big)^2 
\frac{|z_{n+1}|^a}{|z|}\,\d z\,\d t\notag\allowdisplaybreaks\\
&\leq \frac4{r_0}\int_{\sfrac{r_0}2}^{r_1} 
\int_{B_t\setminus B_{\sfrac{t}{2}}}
\Big[\big(\nabla {u_{\underline{0}}}(z) \cdot z - I_{u_{\underline{0}}}(t)\,{u_{\underline{0}}}(z)\big)^2+\big(I_{u_{\underline{0}}}(t) 
- I_{u_{\underline{0}}}(\sfrac{r_0}2)\big)^2 u_{\underline{0}}^2(z)\Big]\frac{|z_{n+1}|^a}{|z|}\,\d z\,\d t\notag
\allowdisplaybreaks\\
& \stackrel{\eqref{e:I mon},\,\eqref{e:monotone additive}}{\leq}
\frac2{r_0}\int_{\sfrac{r_0}2}^{r_1} \frac{t}{2}\, 
H_{u_{\underline{0}}}(t)\,\Big(\frac{\d}{\d t}\big(I_{u_{\underline{0}}}(t)\big)-R_{u_{\underline{0}}}(t)\Big)\, \d t\notag\\
&+ \frac{16}{r_0}\,\Big((I_{u_{\underline{0}}}(r_1)- I_{u_{\underline{0}}}(\sfrac{r_0}2))^2
+(A\,C_{\ref{c:monotone additive}})^2(r_1^{\theta}-(\sfrac{r_0}2)^{\theta})^2\Big)\,
\int_{\sfrac{r_0}2}^{r_1} H_{u_{\underline{0}}}(t)\,\d t\notag\\
& \leq
\frac{r_1}{r_0}H_{u_{\underline{0}}}(r_1)\int_{\sfrac{r_0}2}^{r_1} \,\Big(\frac{\d}{\d t}\big(I_{u_{\underline{0}}}(t)\big)-R_{u_{\underline{0}}}(t)\Big)\, \d t\notag\\
&+ 16\,\frac{r_1}{r_0}H_{u_{\underline{0}}}(r_1)\,\big((I_{u_{\underline{0}}}(r_1) - I_{u_{\underline{0}}}(\sfrac{r_0}2))^2
+(A\,C_{\ref{c:monotone additive}})^2(r_1^{\theta}
-(\sfrac{r_0}2)^{\theta})^2\big)\,.
\end{align}
In the last inequality we have used that, if $\varrho_{\ref{l:monotonia}}$ is sufficiently small, 
then $H_{u_{\underline{0}}}(t) \leq H_{u_{\underline{0}}}(r_1)$ for all 
$t \leq r_1$ by \eqref{e:H2}, and that $\frac{\d}{\d t}\big(I_{u_{\underline{0}}}(t)\big)-R_{u_{\underline{0}}}(t)\geq 0$ thanks to \eqref{e:I mon}. 
Moreover, estimate \eqref{e:Ru} in Proposition~\ref{p:almost monotonicity}, $I_{u_{\underline{0}}}(r_1) \leq A$, the quasi-monotonicity 
of the frequency function and the choice $2r_1\leq\varrho_{\ref{l:monotonia}}$ imply
\[
\int_{\sfrac{r_0}2}^{r_1} |R_{u_{\underline{0}}}(t)|\, \d t \leq 
A\,C_{\ref{p:almost monotonicity}}e^{C_{\ref{p:almost monotonicity}}r_1^{\theta}}\,
\big(r_1^{\theta}-(\sfrac{r_0}2)^{\theta}\big)\,.
\]
Hence, from \eqref{e:prima stima0} we conclude that 
\begin{align*}
\int_{B_{\sfrac{r_1}2}\setminus B_{\sfrac{r_0}2}} &
\Big(\nabla {u_{\underline{0}}}(z) \cdot z - I_{u_{\underline{0}}}(\sfrac{r_0}2)\,{u_{\underline{0}}}(z)\Big)^2 
\frac{|z_{n+1}|^a}{|z|}\,\d z\notag\\ 
& \leq C\,H_{u_{\underline{0}}}(r_1)\, \big(I_{u_{\underline{0}}}(r_1)
+\Lambda\,r_1^\theta-I_{u_{\underline{0}}}(\sfrac{r_0}2)
-\Lambda (\sfrac{r_0}2)^\theta\big),
\end{align*}
where we used that $\sfrac{r_1}{r_0}\leq 8$, and $C>0$.
\end{proof}

\subsection{Oscillation estimate of the frequency}

The following lemma shows how the spatial oscillation of the frequency in 
two nearby points at a given scale is in turn controlled by the radial variations
at comparable scales. 
\begin{proposition}\label{p:D_x frequency}
Let $A,\,\delta>0$. Then there exist
$C_{\ref{p:D_x frequency}}$, $\varrho_{\ref{p:D_x frequency}}>0$ 
such that if 
$\underline{0}\in\ZZ(u)$, 
$\tau\in(0,\sfrac{\varrho_{\ref{p:D_x frequency}}}{144})$ with
$I_{u_{\underline{0}}}(72\tau) \leq A$, then 
\begin{align}\label{e:D_x frequency}
\big\vert I_{u_{x_1}}\big(10\tau\big) - I_{u_{x_2}}\big(10\tau\big)\big\vert
 \leq & \,C_{\ref{p:D_x frequency}} \,
\left[\big(\Delta^{24\tau}_{3\tau}(x_1)\big)^{\sfrac{1}{2}} 
+ \big(\Delta^{24\tau}_{3\tau}(x_2)\big)^{\sfrac{1}{2}}\right]
+C_{\ref{p:D_x frequency}}
\tau^{\theta}\,,
\end{align}
for every $x_1,\,x_2\in B_\tau'\cap\ZZ(u)$. 
\end{proposition}

\begin{proof}
We start off noting that by Remark~\ref{r:freq ben def} and the 
choice $144\tau<\varrho_{\ref{p:D_x frequency}}$, if the constant 
$\varrho_{\ref{p:D_x frequency}}$ is suitably chosen, a simple scaling argument yields that
$I_{u_{x}}(10\tau)$ is well-defined for every $x\in B'_{\sfrac{77\tau}4}$.

To ease the readability of the proof we divide it in several substeps.
\smallskip

\noindent{\bf 1.}
With fixed $x_1,\,x_2\in B_\tau'\cap\ZZ(u)$, let $x_t:=tx_1+(1-t)x_2$, $t\in[0,1]$, and consider the map $t\mapsto I_{u_{x_t}}(10 \tau)$.
The differentiability of the functions $x\mapsto H_{u_x}(10 \tau)$ 
and $x\mapsto D_{u_x}(10 \tau)$ yields that
\[
I_{u_{x_1}}(10 \tau) - I_{u_{x_2}}(10 \tau)=\int_0^1
\frac{\d}{\d t}\big(I_{u_{x_t}}(10 \tau)\big)\,\d t.
\]
Set $e:=x_1-x_2$, then $e \cdot e_{n+1} = 0$; and set for all $y\in\R^{n+1}$
\[
\delta_t (y):=\frac{\d}{\d t}\big(u_{x_t}(x_t+y)\big).
\]
Recalling the very definition of $u_{x_t}$ in \eqref{e:vx0}, it turns out that 
\begin{equation}\label{e:deltat}
\delta_t (y)=\partial_e u(x_t+y)-\partial_e\varphi(x_t+y')+T_{k,x_t}[\partial_e\varphi](x_t+y')
-\mathscr{E}\big[{T}_{k,x_t}[\partial_e\varphi]\big](x_t+y)\,,
\end{equation}
because by linearity (the details of the elementary computations 
are left to the readers)
\begin{equation}\label{e:T derivata}
\frac{\d}{\d t}\big(T_{k,x_t}[\varphi](x_t+y')\big)= 
T_{k,x_t}[\partial_e\varphi](x_t+y')\,,
\end{equation}
and
\[
\frac{\d}{\d t}\big(\mathscr{E}\big[{T}_{k,x_t}[\varphi]\big](x_t+y)\big)=
\mathscr{E}\big[{T}_{k,x_t}[\partial_e\varphi]\big](x_t+y)\,.
\]
Morever, from the very definition of $u_{x_t}$ in \eqref{e:vx0} it is straightforward to prove that
\[
\partial_eu_{x_t}(x_t+y)=\partial_eu(x_t+y)-\partial_e\varphi(x_t+y')
+{T}_{k-1,x_t}[\partial_e\varphi](x_t+y')
-\mathscr{E}\big[{T}_{k-1,x_t}[\partial_e\varphi]\big](x_t+y)\,.
\]
Thus, from \eqref{e:deltat} and the latter equality, by direct calculation it follows that
\begin{equation*}
\delta_t (y)-\partial_e u_{x_t}(x_t+y)=\sum_{|\alpha|=k}D^\alpha(\partial_e\varphi(x_t))
\frac{(y')^\alpha}{\alpha!}-\mathscr{E}\Big(\sum_{|\alpha|=k}D^\alpha(\partial_e\varphi(x_t))
\frac{p_\alpha(\cdot-x_t)}{\alpha!}\Big)(y)\,, 
\end{equation*}
and thus we may conclude that
\begin{equation}\label{e:delta2bis}
|\delta_t (y)-\partial_e u_{x_t}(x_t+y)|\leq C|x_1-x_2||y|^{k}\,.
\end{equation}
Moreover, note also that
\begin{equation}\label{e:delta3}
\nabla\delta_t (y)=\frac{\d}{\d t}\big(\nabla u_{x_t}(x_t+y)\big)\,.
\end{equation}

\noindent{\bf 2.} Thanks to the previous formulas,  for all 
$\lambda\in\R$ we infer
\begin{align}\label{e:dH}
\frac{\d}{\d t}\big(H_{u_{x_t}}(10 \tau)\big) & =- 2 \, \int \dot{\phi}\big(\textstyle{\frac{|y|}{10 \tau}}\big)\,
u_{x_t}(x_t+y)\, \delta_t(y)\,\frac{|y_{n+1}|^a}{|y|} \,\d y\notag\\
& = - 2 \, \int \dot{\phi}\big(\textstyle{\frac{|y|}{10 \tau}}\big)\,
\big(\delta_t(y)-\lambda u_{x_t}(x_t+y)\big)\,u_{x_t}(x_t+y)\,\frac{|y_{n+1}|^a}{|y|} \,\d y
+ 2\lambda H_{u_{x_t}}(10 \tau).
\end{align}
In addition, integrating by parts gives
\begin{align}\label{e:dD}
\frac{\d}{\d t}& \big(D_{u_{x_t}}(10 \tau)\big)  \stackrel{\eqref{e:delta3}}{=}  
2\int \phi\big(\textstyle{\frac{|y|}{10 \tau}}\big)\,
\nabla\delta_t (y)\cdot \nabla u_{x_t}(x_t+y) \,|y_{n+1}|^a\,\d y\notag\\
  = &- \frac 1{5\tau}
 \int \dot{\phi}\big({\textstyle{\frac{|y|}{10 \tau}}}\big)\,
\delta_t(y)\,\nabla u_{x_t}(x_t+y) \cdot y\,
\frac{|y_{n+1}|^a}{|y|} \,\d y
- 2\int \phi\big(\textstyle{\frac{|y|}{10 \tau}}\big)\,
\delta_t(y)\, L_a(u_{x_t}(x_t+y)) \,\d y\notag\\
 = &- \frac 1{5\tau}
\int \dot{\phi}\big(\textstyle{\frac{|y|}{10 \tau}}\big)\,
\big(\delta_t(y)-\lambda u_{x_t}(x_t+y)\big) 
\nabla u_{x_t}(x_t+y) \cdot y\,\frac{|y_{n+1}|^a}{|y|} \,\d y
\notag\\
& +2\lambda G_{u_{x_t}}(10 \tau)  
- 2\int \phi\big(\textstyle{\frac{|y|}{10 \tau}}\big)\,
\delta_t(y)\, L_a(u_{x_t}(x_t+y)) \,\d y\,.
\end{align}
Then, by formula \eqref{e:D} together with \eqref{e:dH} and \eqref{e:dD}, we have that
\begin{align*}
& \frac{\d}{\d t}\big(I_{u_{x_t}}(10 \tau)\big)
 = I_{u_{x_t}}(10 \tau)
 \left(\frac{\frac{\d}{\d t}\big(G_{u_{x_t}}(10 \tau)\big)}{G_{u_{x_t}}(10 \tau)}-\frac{\frac{\d}{\d t}\big(H_{u_{x_t}}(10 \tau)\big)}{H_{u_{x_t}}(10 \tau)}\right)\notag\\
&= -\frac{2}{H_{u_{x_t}}(10 \tau)} \int \dot{\phi}\big(\textstyle{\frac{|y|}{10 \tau}}\big)\,
\big(\delta_t(y)- \lambda u_{x_t}(x_t+y)\big)
\big(\nabla u_{x_t}(x_t+y) \cdot y-I_{u_{x_t}}(10 \tau)\,u_{x_t}(x_t+y)\big)\,\frac{\d\mathfrak{m}(y)}{|y|}\notag\\
& +\frac{10 \tau}{H_{u_{x_t}}(10 \tau)}\int \phi\big(\textstyle{\frac{|y|}{10 \tau}}\big)\,
\Big(u_{x_t}(x_t+y)\, \frac{\d}{\d t}\big(L_a(u_{x_t}(x_t+y))\big)-\delta_t(y)\, L_a(u_{x_t}(x_t+y))\Big) \,\d y\\
& =:J_t^{(1)}+J_t^{(2)}.
\end{align*}
In what follows we estimate separately the two terms $J_t^{(i)}$.
\medskip

\noindent{\bf 3.} We start off with $J_t^{(1)}$. With this aim, first note that 
$I_{u_{\underline{0}}}(10\tau)\leq A\,e^{72^{\theta}
C_{\ref{p:almost monotonicity}}}$ 
by Proposition~\ref{p:almost monotonicity} since 
$144\tau< \varrho_{\ref{p:D_x frequency}}$, provided 
the latter is small enough.
In turn, as $x_t\in B_\tau'$, by \eqref{e:I limitato} in Lemma~\ref{l:lim uniforme} we infer that 
$I_{u_{x_t}}(10\tau)\leq C_{\ref{l:lim uniforme}}+A\,e^{72^{\theta}C_{\ref{p:almost monotonicity}}}$.

We estimate separately the factors of the integrand defining $J_t^{(1)}$
(setting  $x_t+y=z$). We start off with the first one as follows
\begin{align*}
 \big|\delta_t (z-x_t) & - \lambda\,u_{x_t}(z)\big|\leq  
 \big|\partial_e u_{x_t}(z) -  \lambda\,u_{x_t}(z)\big|
 +\big|\delta_t (z-x_t)- \partial_e u_{x_t}(z)\big|\notag\\ 
& \stackrel{\eqref{e:delta2bis}}{\leq}\big|\partial_e u_{x_t}(z) -  \lambda\,u_{x_t}(z)\big| +C|x_1-x_2||z-x_t|^k,
\end{align*}
with $C=C(n,k)>0$. Moreover,  by choosing 
$\lambda:={I_{u_{x_2}}(10 \tau)}-{I_{u_{x_1}}(10 \tau)}$, we infer
\begin{align*}
\big|\partial_e u_{x_t}(z) & -  \lambda\,u_{x_t}(z)\big|=
\big|\nabla u_{x_t}(z) \cdot e  -  \lambda\,u_{x_t}(z)\big|\notag\\
 \leq &
 \big|\nabla u_{x_t}(z)\cdot(z-x_1)- I_{u_{x_1}}(10 \tau)\,u_{x_t}(z)\big|+
 \big|\nabla u_{x_t}(z)\cdot(z-x_2)- I_{u_{x_2}}(10 \tau)\,u_{x_t}(z)\big|
 \notag\\ \leq &
 \big|\nabla u_{x_1}(z)\cdot(z-x_1)- I_{u_{x_1}}(10 \tau)\,u_{x_1}(z)\big|+
 \big|\nabla u_{x_2}(z)\cdot(z-x_2)- I_{u_{x_2}}(10 \tau)\,u_{x_2}(z)\big|
 \notag\\
 &+\big|\nabla(u_{x_t}(z)- u_{x_1}(z))\cdot(z-x_1)- I_{u_{x_1}}(10 \tau)\,(u_{x_t}(z)-u_{x_1}(z))\big|\notag\\
 &+\big|\nabla(u_{x_t}(z)- u_{x_2}(z))\cdot(z-x_2)- I_{u_{x_2}}(10 \tau)\,(u_{x_t}(z)-u_{x_2}(z))\big|\,.
\end{align*}
Using inequalities \eqref{e:stima rescal}-\eqref{e:stima grad rescal}
in Remark~\ref{r:rescal}, we estimate the last two addends as follows
\begin{align*}
\big|\nabla(u_{x_t}(z) & - u_{x_i}(z))\cdot(z-x_i)- I_{u_{x_i}}(10 \tau)\,(u_{x_t}(z)-u_{x_i}(z))\big|
\leq C\tau^{k+1}\,,
\end{align*}
for some constant $C>0$, for $i=1,2$.
In the last inequality we have used that $|z-x_i|\leq 12\tau$, being $z\in B_{10 \tau}(x_t)$.
Therefore, we have 
\begin{align}\label{e:polinomi}
 \big|\delta_{t}(z-x_t) & - \lambda\,u_{x_t}(z)\big|\leq 
 \big|\nabla u_{x_1}(z)\cdot(z-x_1)- I_{u_{x_1}}(10 \tau)\,u_{x_1}(z)\big|
 \notag\\ & 
 +\big|\nabla u_{x_2}(z)\cdot(z-x_2)- I_{u_{x_2}}(10 \tau)\,u_{x_2}(z)\big|+ C\,
 \tau^{k+1}=:\psi(z).
\end{align}

For the second factor, we note that for $i=1,2$
\begin{align*}
 |\nabla u_{x_t}(z)\cdot (z-x_t)-&I_{u_{x_t}}(10 \tau)u_{x_t}(z)|\notag\\ 
 \leq & |\nabla u_{x_i}(z)\cdot (z-x_t)-I_{u_{x_i}}(10 \tau)u_{x_i}(z)|
+|\nabla \big(u_{x_t}(z)-u_{x_i}(z)\big)\cdot (z-x_t)|\notag\\ 
& +|I_{u_{x_t}}(10 \tau)||u_{x_i}(z)-u_{x_t}(z)|
+|I_{u_{x_i}}(10 \tau)-I_{u_{x_t}}(10 \tau)||u_{x_i}(z)|\notag\\
 \leq & 
\tau^{k+1}+C\,|u_{x_i}(z)|.
\end{align*}
To estimate the last three addends we have used the very definition of 
$u_{x_t}$ in \eqref{e:vx0}, formula \eqref{e:I limitato} and inequalities
\eqref{e:stima rescal}-\eqref{e:stima grad rescal} in Remark~\ref{r:rescal}, 
taking into account that 
$|z-x_i|\leq 12\tau$ being $z\in B_{10 \tau}(x_t)$. 
Therefore, we get
\begin{align}\label{e:M}
&|\nabla u_{x_t}(z)\cdot (z-x_t)-I_{u_{x_t}}(10 \tau)u_{x_t}(z)|\leq 
\psi(z)+C\,(|u_{x_1}(z)|+|u_{x_2}(z)|)\,.
\end{align}
By collecting \eqref{e:polinomi} and \eqref{e:M}, using H\"older
inequality we conclude that there exists $C>0$ 
\begin{align}\label{e:derivata2.5}
J_t^{(1)} &\leq   \frac{C}{H_{u_{x_t}}(10 \tau)} 
\int-\dot{\phi}\big(\textstyle{\frac{|z-x_t|}{10 \tau}}\big)
\psi(z)\big(\psi(z)+|u_{x_1}(z)|+|u_{x_2}(z)|\big)
\frac{|z_{n+1}|^a}{|z-x_t|}\,\d z\notag\\
&\leq
\frac{C}{H_{u_{x_t}}(10 \tau)} \Big(
\int-\dot{\phi}\big(\textstyle{\frac{|z-x_t|}{10 \tau}}\big)
\psi^2(z)\frac{|z_{n+1}|^a}{|z-x_t|}\,\d z\Big)^{\sfrac12}
\cdot\notag\\& \hskip1cm
\cdot\Big(\int-\dot{\phi}\big(\textstyle{\frac{|z-x_t|}{10 \tau}}\big)
\big(\psi^2(z)+|u_{x_1}(z)|^2+|u_{x_2}(z)|^2\big)
\frac{|z_{n+1}|^a}{|z-x_t|}\,\d z\Big)^{\sfrac12}\,.
\end{align}
Clearly, we have that
\begin{align*}
&\int-\dot{\phi}\big(\textstyle{\frac{|z-x_t|}{10 \tau}}\big)
\psi^2(z)\frac{|z_{n+1}|^a}{|z-x_t|}\,\d z\notag\\
&\leq
C\int_{B_{10 \tau}(x_t)\setminus B_{5\tau}(x_t)}
|\nabla u_{x_1}(z)\cdot (z-x_1)-I_{u_{x_1}}(10 \tau)u_{x_1}(z)|^2
\frac{|z_{n+1}|^a}{|z-x_t|}\,\d z\notag\\&
+C
\int_{B_{10 \tau}(x_t)\setminus B_{5\tau}(x_t)}
|\nabla u_{x_2}(z)\cdot (z-x_2)-I_{u_{x_2}}(10 \tau)u_{x_2}(z)|^2
\frac{|z_{n+1}|^a}{|z-x_t|}\,\d z+C
\tau^{n+a+2(k+1)}\notag\\ &
\leq C\int_{B_{12 \tau}(x_1)\setminus B_{3\tau}(x_1)}
\big|\nabla u_{x_1}(z)\cdot(z-x_1)- I_{u_{x_1}}(10 \tau)\,u_{x_1}(z)\big|^2\,\frac{|z_{n+1}|^a}{|z-x_1|}\,\d z\notag\\
&+C\int_{B_{12 \tau}(x_2)\setminus B_{3\tau}(x_2)}
\big|\nabla u_{x_2}(z)\cdot(z-x_2)- I_{u_{x_2}}(10 \tau)\,u_{x_2}(z)\big|^2\,\frac{|z_{n+1}|^a}{|z-x_2|}\,\d z+C
\tau^{n+a+2(k+1)}\notag\\ &
\leq C H_{u_{x_1}}(24\tau)\Delta_{\sfrac{3\tau}2}^{24\tau}(x_1)
+CH_{u_{x_2}}(24\tau)\Delta_{\sfrac{3\tau}2}^{24\tau}(x_2)
+C
\tau^{n+a+2(k+1)}\,.
\end{align*}
In the second inequality, we have used that 
$B_{10 \tau}(x_t)\setminus B_{5\tau}(x_t)\subset B_{12 \tau}(x_i)\setminus B_{3\tau}(x_i)$ for $t\in[0,1]$, and that 
$|z-x_i|\leq 2|z-x_t|$ as $z\in B_{10 \tau}(x_t)\setminus B_{5\tau}(x_t)$,
 $i=1,2$.
Moreover, in the third inequality we have applied estimate \eqref{e:monotonia con resto} in Lemma~\ref{l:monotonia} 
to $x_1,\,x_2\in B_\tau'\cap\mathscr{Z}_{\varphi,\theta,\delta}(u)$, with $r_1=24\tau$ and $r_0=3\tau$.
Furthermore, thanks to Corollary~\ref{c:H}, we conclude
\begin{align}\label{e:stima psi}
&\int-\dot{\phi}\big(\textstyle{\frac{|z-x_t|}{10 \tau}}\big)
\psi^2(z)\frac{|z_{n+1}|^a}{|z-x_t|}\,\d z\notag\\
&\leq C\,H_{u_{x_1}}(10\tau)\Delta_{\sfrac{3\tau}2}^{24\tau}(x_1)
+C\,H_{u_{x_2}}(10\tau)\Delta_{\sfrac{3\tau}2}^{24\tau}(x_2)
+C
\tau^{n+a+2(k+1)}\,.
\end{align}
In addition, thanks to \eqref{e:L2 vs H} and $|z-x_t|\geq 5\tau$ we get
\begin{align}\label{e:stima psi2}
\int-\dot{\phi}\big({\textstyle{\frac{|z-x_t|}{10 \tau}}}\big)&
(|u_{x_1}(z)|^2+|u_{x_2}(z)|^2)
\frac{|z_{n+1}|^a}{|z-x_t|}\,\d z\notag\\
&\leq \frac2{5\tau}\big(\|u_{x_1}\|^2_{L^2(B_{10 \tau},\dm)}+\|u_{x_2}\|^2_{L^2(B_{10 \tau},\dm)}\big)
\leq 4H_{u_{x_1}}(10\tau)+4H_{u_{x_2}}(10\tau)\,.
\end{align}
By collecting \eqref{e:derivata2.5}-\eqref{e:stima psi2} we conclude that for some $C>0$
\begin{align*}
J_t^{(1)}\leq & \frac{C}{H_{u_{x_t}}(10 \tau)}
\Big(H_{u_{x_1}}(10\tau)\Delta_{\sfrac{3\tau}2}^{24\tau}(x_1)
+H_{u_{x_2}}(10\tau)\Delta_{\sfrac{3\tau}2}^{24\tau}(x_2)
+
\tau^{n+a+2(k+1)}\Big)
\notag\\ &
+\frac{C}{H_{u_{x_t}}(10 \tau)}\Big(H_{u_{x_1}}(10\tau)+H_{u_{x_2}}(10\tau)\Big)^{\sfrac12}\cdot
\notag\\&\hskip2cm\cdot
\Big(H_{u_{x_1}}(10\tau)\Delta_{\sfrac{3\tau}2}^{24\tau}(x_1)
+H_{u_{x_2}}(10\tau)\Delta_{\sfrac{3\tau}2}^{24\tau}(x_2)
+
\tau^{n+a+2(k+1)}\Big)^{\sfrac12}\notag\\ 
\leq & C\,\Big(\Delta_{\sfrac{3\tau}2}^{24\tau}(x_1)
+\big(\Delta_{\sfrac{3\tau}2}^{24\tau}(x_1)\big)^{\sfrac12}\Big)
+C\,\Big(\Delta_{\sfrac{3\tau}2}^{24\tau}(x_2)
+\big(\Delta_{\sfrac{3\tau}2}^{24\tau}(x_2)\big)^{\sfrac12}\Big)
+C\,
\frac{\tau^{2\theta}}{\delta}+C\,
\frac{\tau^{\theta}}{\delta^{\sfrac12}}\,,
 \end{align*}
where, in the last inequality, we have used Lemma~\ref{l:lim uniforme} 
and that 
$x_1,\,x_2\in B_\tau'\cap\mathscr{Z}_{\varphi,\theta,\delta}(u)$.

Finally, in view of the very definition of the spatial oscillation of the frequency and Corollary~\ref{c:monotone additive}, 
we deduce for some constant depending on $A$ that
\begin{align}\label{e:J1t}
J_t^{(1)}\leq  C\,\Big(
\big(\Delta_{\sfrac{3\tau}2}^{24\tau}(x_1)\big)^{\sfrac12}
+\big(\Delta_{\sfrac{3\tau}2}^{24\tau}(x_2)\big)^{\sfrac12}\Big)
+C\,
\tau^{\theta}\,. 
\end{align}
\smallskip

\noindent{\bf 4.} We estimate next $J_t^{(2)}$. We start off noting that for all $y\in B_1\setminus B_1'$ 
(cf. \eqref{e:diff pol funz bis})
\begin{align}\label{e:esti}
 \frac{\d}{\d t}\big(L_a(u_{x_t}(x_t+y))\big)
& =|y_{n+1}|^a\triangle\Big(\frac{\d}{\d t}\big(T_{k,x_t}[\varphi](x_t+y')-\varphi(x_t+y')\big)\Big)\notag\\
& \stackrel{\eqref{e:T derivata}}{=}|y_{n+1}|^a\triangle\Big(T_{k,x_t}[\partial_e\varphi](x_t+y')-\partial_e\varphi(x_t+y')\Big)
\notag\\&
\leq C\,
|x_1-x_2||y_{n+1}|^a|y'|^{k-2}\leq C\,
\tau|y_{n+1}|^a|y'|^{k-2}.
\end{align}
Then, arguing as in \eqref{e:differenza polinomi}, thanks to  
estimate \eqref{e:stima rescal} in Remark~\ref{r:rescal}, 
we get as $k\geq 2$ 
\begin{align*}
\Big|\int \phi\big(\textstyle{\frac{|y|}{10 \tau}}\big)\, &
\big(u_{x_t}(x_t+y)\, \frac{\d}{\d t}\big(L_a(u_{x_t}(x_t+y))\big)
\,\d y\Big|\\ 
& \stackrel{\eqref{e:esti}}{\leq}  C\,
\tau^{k-1}\int \phi\big(\textstyle{\frac{|y|}{10 \tau}}\big)\,
\big|u_{x_t}(x_t+y)\big||y_{n+1}|^a\,\d y\\ & 
= C\,
\tau^{k-1}\int_{B_{10 \tau}(x_t)} \phi\big(\textstyle{\frac{|z-x_t|}{10 \tau}}\big)\,
\big|u_{x_t}(z)\big||z_{n+1}|^a\,\d z\\ &
\stackrel{\eqref{e:stima rescal}}{\leq} C\,
\tau^{n+a+2k+1}+C\,
\tau^{k-1}\int_{B_{40 \tau}(x_1)} 
\phi\big(\textstyle{\frac{|z-x_1|}{40 \tau}}\big)\,
\big|u_{x_1}(z)\big||z_{n+1}|^a\,\d z\notag\\ 
& \stackrel{\eqref{e:norma L2 D}}{\leq}C\,
\tau^{n+a+2k+1}+C\,
\tau^{\frac{n+a+1}2+k}D_{u_{x_1}}^{\sfrac 12}(40 \tau)\,.
\end{align*}
In addition, \eqref{e:Lav} and \eqref{e:delta2bis} yield
\begin{align*}
\Big|\int \phi\big(\textstyle{\frac{|y|}{10 \tau}}\big)\,& \delta_t(y)\, L_a(u_{x_t}(x_t+y))\,\d y\Big|\\
& \leq C
 \tau^{n+a+2k+1}+
\Big|\int \phi\big(\textstyle{\frac{|y|}{10 \tau}}\big)\,\partial_e u_{x_t}(x_t+y)\, L_a(u_{x_t}(x_t+y))\,\d y\Big|\\
&\leq C
\tau^{n+a+2k+1} 
+
\tau^{k}\int \phi\big(\textstyle{\frac{|y|}{10 \tau}}\big)\,|\nabla u_{x_t}(x_t+y)|\,|y_{n+1}|^a\d y\\
& \leq C
\tau^{n+a+2k+1} +C
\tau^{\frac{n+a+1}2+k}D_{u_{x_t}}^{\sfrac 12}(10 \tau).
\end{align*}
Therefore, by applying repeatedly Lemma~\ref{l:lim uniforme},
by taking into account \eqref{e:frequencies comparison} and by choosing $\varrho_{\ref{p:D_x frequency}}$
sufficiently small, we infer that
\begin{align}\label{e:J2t}
 J_2^{(t)} \leq & \frac{C
 }{H_{u_{x_t}}(10 \tau)}\big(\tau^{n+a+2(k+1)}
 +\tau^{\frac{n+a+1}2+k+1}D_{u_{x_1}}^{\sfrac 12}(40\tau)
 +\tau^{\frac{n+a+1}2+k+1}D_{u_{x_t}}^{\sfrac 12}(10 \tau)\big)
 \notag\\ 
& \stackrel{\eqref{e:Hgrowth}}{\leq} C 
\Big(\frac{\tau^{2\theta}}\delta+
\frac{\tau^{\theta}}{\delta^{\sfrac12}}\Big(\frac{\tau D_{u_{x_1}}(40\tau))}{H_{u_{x_1}}(10 \tau)}\Big)^{\sfrac 12}
+\frac{\tau^{\theta}}{\delta^{\sfrac12}}I_{u_{x_t}}^{\sfrac12}(10 \tau)\Big)\notag\\
& \stackrel{\eqref{e:monotonia H}}{\leq} C
\Big(\tau^{2\theta}+
\tau^{\theta} I_{u_{x_1}}^{\sfrac 12}(40 \tau)
+\tau^{\theta} I_{u_{x_t}}^{\sfrac12}(10 \tau) \Big)
\leq C
\tau^{\theta}\,,
\end{align}
since $144\tau<\varrho_{\ref{p:D_x frequency}}$.  
\medskip

The conclusion in \eqref{e:D_x frequency} follows at once from estimates \eqref{e:J1t} and \eqref{e:J2t}.
\end{proof}

%
%
\section{Proof of the main result}\label{ss:mainresult}

\subsection{Mean-flatness}\label{s:mean-flatness}

Here we show a control of the Jones' $\beta$-number by the oscillation of the frequency. 
Given a Radon measure $\mu$ in $\R^{n+1}$, 
for every $x_0 \in \R^n$ and for every $r>0$, we set
\begin{equation}\label{e:beta}
\beta_\mu(x_0,r) := \inf_{\cL} \Big(
r^{-n-1} \int_{B_r(x_0)} \dist^2(y,\cL)\d\mu(y)\Big)^{\sfrac{1}{2}},
\end{equation}
where the infimum is taken among all affine $(n-1)$-dimensional planes $\cL \subset \R^{n+1}$.

If $x_0 \in \R^{n+1}$ and $r>0$ is such that $\mu(B_{r}(x_0)) >0$, set
$\bar x_{x_0,r}$ the barycenter of $\mu$ in $B_r(x_0)$, {\ie}
\[
\bar x_{x_0,r} := \frac{1}{\mu(B_{r}(x_0))} \int_{B_{r}(x_0)} x \, \d\mu(x)\,,
\]
and 
\[
{\bf B}_{x_0}(v, w) := \int_{B_{r}(x_0)} \big((x-\bar x_{x_0,r}) \cdot v\big)\;\big( 
(x-\bar x_{x_0,r}) \cdot w\big)\,\d\mu(x) 
\quad\forall \; v,\,w \in \R^{n+1}.
\]
Then
\begin{equation}\label{e:beta-charac}
\beta_{\mu}(x_0, r) = \Big(r^{-n-1}\big(\lambda_{n}+\lambda_{n+1}\big)\Big)^{\frac12},
\end{equation}
where $0\leq \lambda_{n+1} \leq \lambda_{n} \leq \cdots \leq \lambda_1$ are the 
eigenvalues of the positive semidefinite bilinear form 
${\bf B}_{x_0}$.

\begin{proposition}\label{p:mean-flatness vs freq}
Let $A,\,\delta>0$. Then there exist constants 
$C_{\ref{p:mean-flatness vs freq}}$, 
$\varrho_{\ref{p:mean-flatness vs freq}}>0$ with this property. 
Let $122r\leq\varrho_{\ref{p:mean-flatness vs freq}}$, 
$\underline{0} \in \ZZ(u)$ and $I_{u_{\underline{0}}}(66r) \leq A$. 
Let $\mu$ be a finite Borel measure with 
$\spt(\mu)\subseteq \ZZ(u)$. Then, for all points 
$p \in B_r'\cap\ZZ(u)$, we have
\begin{equation}\label{e:mean-flatness vs freq}
\beta_{\mu}^2 (p,r) \leq 
\frac{C_{\ref{p:mean-flatness vs freq}}}{r^{n-1}}\left(
\int_{B_{r}(p)}\Delta_{\sfrac{5}{2}r}^{24\,r}(x)\,\d\mu(x)+
r^{2\theta}\mu(B_{r}(p))\right).
\end{equation}
\end{proposition}
\begin{proof}
The proof is a variant of the \cite[Proposition~4.2]{FoSp17}, which in turn follows closely the original 
arguments by Naber and Valtorta in \cite{NaVa1, NaVa2}, therefore we only highlight the main differences.

Without loss of generality assume that
$p \in  B_{r}'\cap\Gamma_\varphi(u)\cap\ZZ(u)$ is such that 
$\mu(B_r(p)) >0$ (otherwise, there is nothing to prove).
Let $\{v_1, \ldots, v_{n+1}\}$ be any diagonalizing basis for the
bilinear form ${\bf B}_p$ introduced in \S~\ref{s:mean-flatness},
with corresponding eigenvalues $0\leq\lambda_{n+1} \leq \lambda_{n} \leq \cdots \leq\lambda_1$.

Since $\spt(\mu) \subset \Gamma_\varphi(u)\subset\R^n\times\{0\}$, we may assume that $v_{n+1} = e_{n+1}$, 
$\lambda_{n+1} = 0$, so that $\beta_{\mu} (p,r) = (r^{-n-1}\lambda_{n})^{\sfrac12}$ by 
 \eqref{e:beta-charac}. Clearly, we may also assume that 
 $\lambda_n>0$.
 
From the very definitions of ${\bf B}_p$ and of its barycenter
we deduce 
\begin{align}\label{e:intermedio}
r^{n+1}\beta_{\mu}^2(p,r) &\int_{B_{11r}(p)
\setminus B_{10r}(p)}|\nabla' u_p(z)|^2|z_{n+1}|^a\d z
\notag\\ & 
\leq n\int_{B_{r}(p)} \int_{B_{12r}(x) \setminus 
B_{9r}(x)}
 \big((z-x) \cdot \nabla u_p(z) - \alpha\,u_p(z)\big)^2
 |z_{n+1}|^a\d z\;\d\mu(x),
\end{align}
where
\[
 \alpha:= \frac{1}{\mu(B_{r})} \int_{B_{r}(p)} 
 I_{u_x}(9r)\d\mu(x).
 \]
Next we estimate the two sides of \eqref{e:intermedio}.

For estimating the left hand side of \eqref{e:intermedio}, we can show by compactness 
that
\begin{align}
D_{u_p}(12r)
\leq C\int_{B_{11r}(p)\setminus B_{10r}(p)}
|\nabla' u_p(z)|^2|z_{
n+1}|^a\d z.
\label{e:stima dal basso}
\end{align}
Here we use the same contradiction argument in \cite[Proposition~4.2]{FoSp17} using
the compactness given by  Corollary~\ref{c:compactness}.

For what concerns the right hand side of \eqref{e:intermedio} 
we proceed as follows. By the triangular inequality we have that
\begin{align}\label{e:intermedio1} 
&\text{r.h.s. of \eqref{e:intermedio}}\notag\\
& \leq 4n\int_{B_{r}(p)}\int_{B_{12r}(x) \setminus 
B_{9r}(x)}
\big((z-x) \cdot \nabla u_x(z) - 
I_{u_x}(9r)\,u_x(z)\big)^2|z_{n+1}|^a\,\d z\;\d\mu(x)\notag\\
& +4n\int_{B_{r}(p)}\int_{B_{12r}(x) \setminus B_{9r}(x)}
\big((z-x) \cdot \nabla (u_x-u_p)(z) - \alpha\,
(u_x-u_p)(z)\big)^2|z_{n+1}|^a\,\d z\;\d\mu(x)\notag\\
& + 4n \int_{B_{r}(p)}
\int_{B_{12r}(x) \setminus B_{9r}(x)}
\Big(I_{u_x}(9r) - \alpha \Big)^2 \,u_x^2(z)|z_{n+1}|^a\,\d 
z\;\d\mu(x)=:J_1+J_2+J_3.
\end{align}
The addends $J_1$ and $J_3$ can be treated as in \cite[Proposition~4.2]{FoSp17}. Indeed, for $J_1$ 
we use Lemma~\ref{l:lim uniforme} and Lemma~\ref{l:monotonia} for a suitable choice of the constants to get 
\begin{align}\label{e:stima dall'alto1}
J_1&\leq C\,r \,
\int_{B_r(p)}H_{u_x}\big(24r\big)
\Delta^{24r}_{9r}(x)\,\d\mu(x)
\leq C\,r \,H_{u_p}(12r)\int_{B_r(p)}\Delta^{24r}_{9r}(x)\,\d\mu(x)\,.
\end{align}
For $J_3$ we use Jensen's inequality, 
Proposition~\ref{p:D_x frequency}, Fubini's Theorem,
inequality \eqref{e:L2 vs H} and \eqref{e:H limitato} in
Lemma~\ref{l:lim uniforme} to get
\begin{align}
J_3 \leq C\,r\,H_{u_p}(12r)\Big(
\int_{B_{r}(p)}\Delta^{22r}_{\sfrac{5}{2}r}(x)\, 
\d \mu(x)+r^{2\theta}\mu\big(B_{r}(p)\big)\Big).
\label{e:stima dall'alto2}
\end{align}
Note that the extra term with respect to \cite[Proposition~4.2]{FoSp17}
arises as a consequence of the additional error term in Proposition~\ref{p:D_x frequency}.

To estimate $J_2$ in \eqref{e:intermedio1} we first note that
$\nabla\big(T_{k,x}[\varphi]\big)=T_{k-1,x}[\nabla\varphi]$.
Then, we use estimates \eqref{e:stima rescal} and \eqref{e:stima grad rescal} in Remark~\ref{r:rescal} 
to deduce that for all $x\in B_r(p)$ and $z\in B_{12r}(x) \setminus B_{9r}(x)$ we have
\begin{align*}
&\big((z-x) \cdot \nabla (u_x-u_p)(z) - \alpha\,
(u_x-u_p)(z)\big)^2\notag\\
&\leq C\,
\Big(r^2|\nabla\big(T_{k,x}[\varphi](z)-T_{k,p}[\varphi](z)\big)|^2
+\alpha^2|T_{k,x}[\varphi](z)-T_{k,p}[\varphi]|^2(z)\Big)
  \leq C\,
  r^{2(k+1)}\,.
\end{align*}
Therefore, integrating the last estimate we conclude that
\begin{align}\label{e:stima dall'alto3}
J_2 \leq  C\,
r^{n+a+2k+3}\mu(B_{r}(p))\,.
\end{align}
We can now collect the estimates \eqref{e:stima dal basso}--\eqref{e:stima dall'alto3} and 
use Corollary~\ref{c:monotone additive} to get
\begin{align*}
r^{n+1}&\beta_{\mu}^2(p,r) D_{u_p}(12r)\\
&\leq C\,r \,H_{u_p} (12r)
\int_{B_{r}(p)}\big(\Delta^{24r}_{9r}(x)+\Delta^{22r}_{\sfrac52r}(x)\big)\, \d \mu(x)
\\&+ C\,r^{1+2\theta}\mu(B_{r}(p)) H_{u_p}(12r) + C\,r^{n+a+2k+3}\mu(B_{r}(p))\\
&\leq
C\,r \,H_{u_p}(12r)\,
\Big(\int_{B_{r}(p)}\Delta^{24r}_{\sfrac{5}{2}r}(x)\, \d \mu(x)
+\big(r^{2\theta}+\frac{r^{n+a+2(k+1)}}{H_{u_p}(12r)}\big)\mu(B_{r}(p))\Big).
\end{align*}
Finally, by assumption $p\in\ZZ$, then $e^{C_{\ref{p:almost monotonicity}}\normphi (12r)^\theta}I_{u_p}(12r) 
\geq 1+s$ (cf. Proposition~\ref{p:almost monotonicity}, Corollary~\ref{c:minima freq} and the choice 
$122r\leq 1$), so that the upper inequality in \eqref{e:frequencies comparison} yields 
\eqref{e:mean-flatness vs freq}.
\end{proof}

\subsection{Rigidity of homogeneous solutions}\label{s:rigidity}
In this section we extend the results on the rigidity of almost
homogeneous solutions established in \cite{FoSp17}.

We denote by $\cH_\lambda$ the space of all non-zero
$\lambda$-homogeneous solutions to the thin obstacle problem 
\eqref{e:ob-pb local} with zero obstacle,
\[
\cH_{\lambda} := \Big\{ u\in H^1_\loc(\R^{n+1},\dm)\setminus\{0\}
:\, u(x) = |x|^\lambda\,u\big(\sfrac{x}{|x|}\big),\text{ 
$u\vert_{B_1}$ solves \eqref{e:ob-pb local} with $\ph\equiv 0$}\Big\},
\]
and set $\cH:= \bigcup_{\lambda\geq 1+s}\cH_{\lambda}$.
The spine $S(u)$ of $u \in \cH$
is the maximal subspace of invariance of $u$,
\[
S(u) := \Big\{ y\in\R^n\times \{0\}\;:\; u(x+y) = u(x) \quad \forall\; x\in 
\R^{n+1}\Big\}.
\]
As observed in \cite{FoSp17}, the maximal dimension of the spine
of a function  in $\cH$ is at most $n-1$ and we set 
$u\in\cH^\top$ if $u \in \cH$ and $\dim S(u) = n-1$,
and $\cH^\low:=\cH\setminus \cH^\top$.
All functions in $\cH^\top$ are classified in \cite[Lemma~5.3]{FoSp17}.
Note also that by Caffarelli, Salsa and Silvestre~\cite{CaSaSi08}
\begin{equation}\label{e:H1+s}
 \cH_{1+s}\subseteq\cH^{top}.
\end{equation}

We next introduce the notion of almost homogeneous solutions.
Given $\delta>0$ and $x_0\in \ZZ$ we set
\[
J_{u_{x_0}} (t) := e^{C_{\ref{p:almost monotonicity}}
t^\theta} I_{u_{x_0}}(t)
\quad \forall\; t\in (0, \varrho_{\ref{p:almost monotonicity}}]\,.
\]
\begin{definition}\label{d:almost hom}
Let $\eta>0$ and let $u:B_1\to\R$ be a solution to thin obstacle problem \eqref{e:ob-pb local} with obstacle $\ph$ (as usual 
$\|\ph\|_{C^{k+1}(B_1')}\leq1$).
Assume that $0\in \ZZ$ and $\varrho\leq \varrho_{\ref{p:almost monotonicity}}$, $u$ is called \textit{$\eta$-almost 
homogeneous in $B_\varrho$} if
\[
J_{u_{\underline{0}}}(\sfrac\varrho2) - J_{u_{\underline{0}}}(\sfrac\varrho4)\leq \eta.
\]
\end{definition}

The following lemma justifies this terminology
and it is the analog of \cite[Lemma~5.5]{FoSp17}.

\begin{lemma}\label{l:almost hom}
Let $\eps,\,A>0$. There exists 
$\eta_{\ref{l:almost hom}}>0$ with the following property: 
for every $\delta>0$ there exists $\varrho_{\ref{l:almost hom}}$ such that, if $u$ is a $\eta_{\ref{l:almost hom}}$-almost homogeneous solution of \eqref{e:ob-pb local} in $B_\varrho$  with $\varrho\leq \varrho_{\ref{l:almost hom}}$ and obstacle $\ph$,  $\underline{0}\in\ZZ(u)$ and 
$I_{u_{\underline{0}}}(\varrho_{\ref{l:almost hom}})\leq A$, then 
\begin{align}\label{e:almost hom}
\big\| u_{\underline{0},\varrho} - w\big\|_{H^1(B_{\sfrac14}, \dm)}
\leq \eps,
\end{align}
for some homogeneous solution $w\in \cH$.
\end{lemma}
\begin{proof}
The proof follows by a contradiction argument similar to \cite[Lemma~5.5]{FoSp17}. 
Assume that for some $\eps,A>0$ we could find sequences of 
numbers $\delta_l, \varrho_l$ and of $\sfrac1l$-almost 
homogeneous solutions $u_l$ of \eqref{e:ob-pb local} in $B_{\varrho_l}$, with $\varrho_l$ arbitrarily small, such that 
$\underline{0}\in\ZZ(u_l)$ and
\begin{equation}\label{e.lontano}
\inf_l\inf_{w\in\cH} 
\big\| (u_l)_{\underline{0},\varrho_l} - w\big\|_{H^1(B_{\sfrac14},\dm)}\geq \eps\,,
\end{equation}
and satisfying the bounds $I_{(u_l)_{\underline{0}}}(\varrho_l)\leq A$. 

Consider $v_l:=(u_l)_{\underline{0},\varrho_l}$, then by 
Corollary~\ref{c:compactness} applied to $v_l$ 
there would be a subsequence, not relabeled, 
converging in $H^1(B_1, \dm)$ 
to a solution $v_\infty$ of the thin obstacle problem 
with zero obstacle. By Proposition~\ref{p:almost monotonicity} 
there is some $A'$ independent of $l$ such that 
$I_{(u_l)_{\underline{0}}}(t)\leq A'$ for all $t\in(0,\varrho_l]$, 
then from \eqref{e:H1} in Corollary~\ref{c:H} we would infer that
\[
-\int{\dot{\phi}(\textstyle{\frac{2|x|}{t}})
\frac{|v_\infty|^2}{|x|}}|x_{n+1}|^a\d x
=-\lim_l\int{\textstyle{\dot{\phi}(\frac{2|x|}{t}})
\frac{|v_l|^2}{|x|}}|x_{n+1}|^a\d x
= \lim_{l}\frac{H_{(u_l)_{\underline{0}}}(\sfrac{\varrho_l}2)}{H_{(u_l)_{\underline{0}}}(\varrho_l)} \geq  2^{-(n+a+2A')},
\]
in turn implying that $v_\infty$ is not zero. 
On the other hand, we would also get
\[
I_{(v_\infty)_{ \underline{0}}}(\sfrac12) - 
I_{(v_\infty)_{ \underline{0}}}\big(\sfrac 14\big)
= \lim_l\big(J_{(v_l)_{ \underline{0}}}(\sfrac12) 
- J_{(v_l)_{ \underline{0}}}(\sfrac14)\big) = 
\lim_l\big(J_{(u_l)_{ \underline{0}}}(\sfrac{\varrho_l}2) 
- J_{(u_l)_{ \underline{0}}}(\sfrac{\varrho_l}4)\big)=
0,
\]
and thus we would conclude that $v_\infty\in \cH$ being 
a solution to the lower dimensional obstacle problem with 
constant frequency (see for instance \cite[Proposition~2.7]{FoSp17}).  
We have thus contradicted \eqref{e.lontano}.
\end{proof}

A rigidity property of the type shown in \cite[Proposition~5.6]{FoSp17} holds for the non-zero obstacle problem.

\begin{proposition}\label{p:rigidity}
Let $A,\,\tau>0$. There exists $\eta_{\ref{p:rigidity}}>0$ with this property.
For every $\delta>0$ there exists $\varrho_{\ref{p:rigidity}}$ such that, if $u$ is a $\eta_{\ref{p:rigidity}}$-almost homogeneous solution of \eqref{e:ob-pb local} in $B_\varrho$  with $\varrho\leq \varrho_{\ref{p:rigidity}}$ and obstacle $\ph$,  $\underline{0}\in\ZZ(u)$ and 
$I_{u_{\underline{0}}}(\varrho_{\ref{p:rigidity}})\leq A$, then 
the following dichotomy holds:
\begin{itemize}
\item[(i)] either for every point $x\in B_{\sfrac\varrho2}'\cap\ZZ(u)$ we have 
\begin{align}\label{e:rigidity1}
\left\vert J_{u_x}(\sfrac\varrho2) 
- J_{u_{\underline{0}}}(\sfrac\varrho2)\right\vert\leq\tau,
\end{align}
\item[(ii)] or there exists a linear subspace 
$V\subset\R^{n}\times\{0\}$ of dimension $n-2$ such that
\begin{align}\label{e:rigidity2}
\begin{cases}
y\in B_{\sfrac\varrho2}'\cap\ZZ(u),\\
J_{u_y}(\sfrac\varrho8) - J_{u_y}(\sfrac\varrho{16})\leq \eta_{\ref{p:rigidity}}
\end{cases}
\quad\Longrightarrow\quad \dist(y,V)<\tau \varrho.
\end{align}
\end{itemize}
\end{proposition}

\begin{proof}
The proof proceeds by contradiction and follows the strategy 
developed in \cite[Proposition~5.6]{FoSp17}.
Let $A,\,\tau>0$ be given constants 
and assume that there exist $\delta_l, \varrho_l$ and a sequence $(u_l)_{l\in\N}$ of 
$\sfrac 1l$-almost homogeneous solutions in $B_{\varrho_l}$ such that 
$\underline{0}\in {\mathscr{Z}_{\varphi_l,\theta,\delta_l}}(u_l)$,
$I_{(u_l)_{\underline{0}}}(\varrho_l)\leq A$ and such that:
\begin{itemize}
\item[(i)] there exists $x_l\in B_{\sfrac{\varrho_l}4}'\cap
{\mathscr{Z}_{\varphi_l,\theta,\delta_l}}(u_l)$ for which 
\begin{align}\label{e:rigidity1-contra}
\left\vert J_{(u_l)_{x_l}}(\sfrac{\varrho_l}2) 
- J_{(u_l)_{\underline{0}}}(\sfrac{\varrho_l}2)\right\vert>\tau,
\end{align}
\item[(ii)] for every linear subspace $V\in\R^n\times\{0\}$
of dimension $n-2$ there exists 
$y_l\in B_{\sfrac{\varrho_l}4}'\cap
\mathscr{Z}_{\varphi_l,\theta,\delta_l}(u_l)$
(a priori depending on $V$) such that
\begin{align}\label{e:rigidity2-contra}
J_{(u_l)_{y_l}}(\sfrac{\varrho_l}8) - J_{(u_l)_{y_l}}(\sfrac{\varrho_l}{16})
\leq\sfrac 1l
\quad\text{and}\quad \dist\big(y_l, V\big) \geq \tau{\varrho_l}.
\end{align}
\end{itemize}
We consider the rescaled functions $v_l:=
(u_l)_{\underline{0},\varrho_l}:B_2\to \R$.
By the compactness result in Corollary~\ref{c:compactness} 
we deduce that, up to passing to a subsequence (not relabeled), 
there exists a nonzero function $v_\infty$ solution to the thin 
obstacle problem \eqref{e:ob-pb local} in $B_1$ with null 
obstacle such that $v_l \to v_\infty$ in $H^1(B_1, \dm)$.
Moreover, $v_\infty \in \mathcal{H}$ thanks to 
Lemma~\ref{l:almost hom}.  

If $v_\infty \in \cH^\top$, then \eqref{e:rigidity1-contra} is 
contradicted. Indeed, up to choosing a further subsequence, 
we can assume that $z_l:=\varrho_l^{-1}x_l\to z_\infty \in \bar B_{\sfrac 12}$.
Note that the points $z_l \in \mathcal{N}(v_l)$, as $x_l\in\Gamma_{\varphi_l}(u_l)$, so that
\[
v_l(z_l) = (u_l)_{\underline{0},\varrho_l}(z_l) = 
\frac{\varrho_l^{\frac{n+a}{2}}}{H_{(u_l)_{\underline{0}}}^{\sfrac12}(\varrho_l)}
(u_l)_{\underline{0}}(x_l)=0\,.
\]
In addition, by \eqref{e:bd condition1} and being $\mathscr{E}\big[T_{k,\underline{0}}[\ph_l]\big]$ 
even with respect to $\{x_{n+1}=0\}$ (cf. Lemma~\ref{l:grufolo polinomio}), for all $l$ we infer that
\[
\lim_{t\downarrow 0}t^a\de_{n+1}v_l(z_l',t) =
\frac{\varrho_l^{\frac{n+a}{2}}}{H_{(u_l)_{\underline{0}}}^{\sfrac12}(\varrho_l)}
\lim_{t\downarrow 0}t^a\de_{n+1}\Big(u_l(x_l',t)  - 
\mathscr{E}\big[T_{k,\underline{0}}[\ph_l]\big](x_l',t)\Big) = 0\,.
\]
Hence, we conclude that $z_\infty \in \mathcal{N}(v_\infty)$ in view of \eqref{e:cpt3}.
Moreover, by taking into account the very definition of $v_l$ and
Remark~\ref{r:scaling} we get by scaling
\begin{align*}
&\left\vert I_{(v_\infty)_{z_\infty}}(\sfrac12) - 
I_{(v_\infty)_{\underline{0}}}(\sfrac12)\right\vert 
=\left|\frac{\sfrac12\int\phi(2|x-z_\infty|)|\nabla v_\infty|^2|x_{n+1}|^a\d x}
 {-\int\dot{\phi}(2|x-z_\infty|)\frac{|v_\infty|^2}{|x-z_\infty|}|x_{n+1}|^a\d x}-
\frac{\sfrac12\int\phi(2|x|)|\nabla v_\infty|^2|x_{n+1}|^a\d x}
 {-\int\dot{\phi}(2|x|)\frac{|v_\infty|^2}{|x|}|x_{n+1}|^a\d x}
 \right|\\
 &=\lim_{l\to+\infty}\left|\frac{\sfrac12\int\phi(2|x-z_l|)|\nabla v_l|^2|x_{n+1}|^a\d x}
 {-\int\dot{\phi}(2|x-z_l|)\frac{|v_l|^2}{|x-z_l|}|x_{n+1}|^a\d x}-
\frac{\sfrac12\int\phi(2|x|)|\nabla v_l|^2|x_{n+1}|^a\d x}
 {-\int\dot{\phi}(2|x|)\frac{|v_l|^2}{|x|}|x_{n+1}|^a\d x}
 \right|\\
 &=\lim_{l\to+\infty}\left|\frac{\sfrac{\varrho_l}2\int\phi(\frac{|z-x_l|}{\sfrac{\varrho_l}2})|\nabla (u_l)_{\underline{0}}|^2|z_{n+1}|^a\d z}
 {-\int\dot{\phi}(\frac{|z-x_l|}{\sfrac{\varrho_l}2})\frac{|(u_l)_{\underline{0}}|^2}{|z-x_l|}|z_{n+1}|^a\d z}-
 \frac{\sfrac{\varrho_l}2 D_{(u_l)_{\underline{0}}}(\sfrac{\varrho_l}2)}{H_{(u_l)_{\underline{0}}}(\sfrac {\varrho_l}2)}\right|
 &\\
&=\lim_{l\to+\infty}
\left|\frac{\sfrac{\varrho_l}2D_{(u_l)_{x_l}}(\sfrac{\varrho_l}2)}{H_{(u_l)_{x_l}}(\sfrac{\varrho_l}2)}
-\frac{\sfrac{\varrho_l}2D_{(u_l)_{\underline{0}}}(\sfrac{\varrho_l}2)}{H_{(u_l)_{\underline{0}}}(\sfrac{\varrho_l}2)}\right|
\stackrel{\eqref{e:frequencies comparison}}{=} \lim_{l\to+\infty}
\left\vert I_{(u_l)_{x_l}}(\sfrac{\varrho_l}2) - I_{(u_l)_{\underline{0}}}(\sfrac{\varrho_l}2)\right\vert\\
&= 
\lim_{l\to+\infty}\left\vert J_{(u_l)_{x_l}}(\sfrac {\varrho_l}2) - J_{(u_l)_{\underline{0}}}(\sfrac{\varrho_l}2)\right\vert\geq 
\tau,
\end{align*}
which is a contradiction to the constancy of the frequency 
at critical points of the homogeneous solution 
$v_\infty\in \cH^{\top}$ (see \cite[Lemma~5.3]{FoSp17}).
The fourth equality is justifed by taking into account that
$x_l\in{\mathscr{Z}_{\varphi_l,\theta,\delta_l}}(u_l)$ (cf. \eqref{e:Hgrowth} and \eqref{e:Ddelta}), 
and in view of estimates \eqref{e:stima rescal} and \eqref{e:stima grad rescal} in Remark~\ref{r:rescal}, 
in turn implying for all $z\in B_{\sfrac{\varrho_l}2}(x_l)$ (recall that $\varrho_l^{-1}x_l\to z_\infty$)
\[
|(u_l)_{\underline{0}}(z)-(u_l)_{x_l}(z)|\leq C\varrho_l^{k+1},\qquad
|\nabla\big((u_l)_{\underline{0}}(z)-(u_l)_{x_l}(z)\big)|
\leq C\varrho_l^k\,.
\]
Moreover, \eqref{e:frequencies comparison} 
can be employed in the last two equalities as $x_l\in B_{\sfrac{\varrho_l}4}'\cap
\mathscr{Z}_{\varphi_l,\theta,\delta_l}(u_l)$. 

Instead, if $v_\infty\in\cH^\low$, we show a contradiction to \eqref{e:rigidity2-contra}
with $V$ any $(n-2)$-dimensional subspace containing $S(v_\infty)$.
Indeed, let $y_l$ be as in \eqref{e:rigidity2-contra} for such a choice 
of $V$. By compactness, up to passing to a subsequence (not relabeled), 
$z_l:=\varrho_l^{-1}y_l\to z_\infty$ for some 
 $z_\infty\in \bar B_{\sfrac12}$ with $\dist(z_\infty, V)\geq \tau$. 
In addition, arguing as before
\begin{align*}
&\left\vert I_{(v_\infty)_{z_\infty}}\big(\sfrac18\big) - 
I_{(v_\infty)_{z_\infty}}\big(\sfrac1{16}\big)\right\vert \\
&=\lim_{l\to+\infty}
\left|\frac{\sfrac{\varrho_l}8D_{(u_l)_{y_l}}(\sfrac{\varrho_l}8)}{H_{(u_l)_{y_l}}(\sfrac{\varrho_l}8)}
-\frac{\sfrac{\varrho_l}{16}D_{(u_l)_{y_l}}(\sfrac{\varrho_l}{16})}{H_{(u_l)_{y_l}}(\sfrac{\varrho_l}{16})}\right|
\stackrel{\eqref{e:frequencies comparison}}{=} 
\lim_{l\to+\infty}
\left\vert I_{(u_l)_{y_l}}(\sfrac{\varrho_l}8) - I_{(u_l)_{y_l}}(\sfrac{\varrho_l}{16})\right\vert\\
&=\lim_{l\to+\infty}\left\vert J_{(u_l)_{y_l}}(\sfrac{\varrho_l}8) - J_{(u_l)_{y_l}}(\sfrac{\varrho_l}{16})\right\vert=0\,.
\end{align*}
Again, note that \eqref{e:frequencies comparison} can be employed since 
$y_l\in B_{\sfrac{\varrho_l}2}'\cap\mathscr{Z}_{\varphi_l,\theta,\delta_l}(u_l)$. 
By \cite[Proposition~2.7, Lemma~5.2]{FoSp17} it follows
that $z_\infty \in S(v_\infty)$, thus contradicting 
$S(v_\infty)\subseteq V$ and $\dist(z_\infty,V)\geq \tau$.
\end{proof}

\subsection{Proof of Theorem~\ref{t:phi Ck}}\label{s:misura}
We start off noting that it suffices to prove that $\Gamma_{\varphi,\theta}(u)\cap \bar{B}_1(x_0)$ 
satisfies all the conclusions for all $x_0\in\Gamma_{\varphi,\theta}(u)$. For all $R\in(0,1)$, 
we can find a finite number of balls $B_{\sfrac R2}(x_i)$, $x_i\in\Gamma_{\varphi,\theta}(u)$ 
for $i\in\{1,\ldots,M\}$, whose union cover 
$\Gamma_{\varphi,\theta}(u)\cap \bar{B}_1(x_0)$. We shall choose
appropriately $R$ in what follows. Moreover, with fixed 
$i\in\{1,\ldots,M\}$, by horizontal translation we may reduce to 
$x_i=\underline{0}\in\Gamma_{\varphi,\theta}(u)$ without loss of generality.

Then, recalling the definition of $\Gamma_{\varphi,\theta}(u)$ in \eqref{e:Hgrowth intro}
we have that
\[
\Gamma_{\varphi,\theta}(u)\cap B_{\sfrac R2}'=\cup_{j\in\N}
\mathscr{Z}_{\varphi,\theta,\sfrac1j}^R(u)\,,
\]
where 
\[
\mathscr{Z}_{\varphi,\theta,\sfrac1j}^R(u) := 
\big\{x_0\in \Gamma_\varphi(u)\cap B_{\sfrac R2}':\, 
H_{u_{x_0}}(r)\geq\sfrac{r^{n+a+2(k+1-\theta)}}j\quad\forall\; r\in(0,\sfrac R2)\big\}
\]
Hence, we may establish the result for 
$\mathscr{Z}_{\varphi,\theta,\sfrac1j}^R(u)$ with $j\in\N$ fixed. 
 
Next, note that as $\underline{0}\in\Gamma(u)$, the function 
\[
\widetilde{u}(y):=u(Ry)-u(\underline{0})
\]
solves the fractional obstacle problem \eqref{e:ob-pb local} in $B_1$ with obstacle function 
$\widetilde{\varphi}(\cdot):=
\varphi(R\cdot)-\varphi(\underline{0})$. 
Moreover, 
$\Gamma_{\widetilde{\varphi},\theta}(\widetilde{u})\cap B_{\sfrac 12}'=\frac1R\big(\Gamma_{\varphi,\theta}(u)\cap B_{\sfrac R2}'\big)$, with $\widetilde{u}_{\sfrac{z}R}(\cdot)=
\,u_{z}(R\cdot)$ if $z\in\Gamma_{\varphi,\theta}(u)\cap B_{\sfrac R2}'$,
being $T_{k,\sfrac{z}R}[\widetilde{\varphi}](\cdot)=
T_{k,z}[\varphi](R\cdot)$.
Thus, we get that $z\in\mathscr{Z}_{\varphi,\theta,\sfrac1j}^R(u)$ 
if and only if $\sfrac{z}R\in B_{\sfrac 12}'
\cap\mathscr{Z}_{\widetilde{\varphi},\theta,\sfrac{R^{2(k+1-\theta)}}j}(\widetilde{u})$.
In addition, it is easy to check that
\[
 \|\widetilde{\varphi}\|_{C^{k+1}(B_1')}
 \leq 
 R\|\nabla \varphi\|_{C^k(B_R',\R^n)}\,.
\]
We choose  $R>0$ sufficiently small so that 
$\|\widetilde{\varphi}\|_{C^{k+1}(B_R')}\leq1$ and the smallness conditions on the radii in all the statements of Sections \ref{ss:cikappagamma}-\ref{ss:mainresult} are satisfied.

In such a case the proof, of the main results can be obtained by following verbatim \cite[Sections~6--8]{FoSp17}. 
Indeed, \cite[Proposition~6.1]{FoSp17}, that leads both to the local finiteness of the Minkowskii content of $\mathscr{Z}_{\widetilde{\varphi},\theta,\delta}(\widetilde{u})$ and to its
$(n-1)$-rectifiability,
is based on a covering argument that exploits the lower bound on the frequency in  Corollary~\ref{c:minima freq}, the control of the mean oscillation via the frequency in Proposition~\ref{p:mean-flatness vs freq}, the rigidity of almost homogeneous solutions in 
Proposition~\ref{p:rigidity}, the discrete Reifenberg theorem by 
Naber and Valtorta \cite[Theorem~3.4, Remark~3.9]{NaVa1}, and the rectifiability criterion either by Azzam and Tolsa  \cite{AzTo15} or by Naber and Valtorta \cite{NaVa1, NaVa2}.
Therefore, the only extra-care needed in the current setting is to start the covering argument from a scale which is small enough to validate the conclusions of the lemmas and propositions of the previous sections.


Finally, the classification of blow-up limits is exactly that stated in
\cite[Theorem~1.3]{FoSp17}, and proved in 
\cite[Section~8]{FoSp17}, in view of Lemma~\ref{l:freq lower} 
and Corollary~\ref{c:compactness}.
\qedhere

%
%

\bibliographystyle{plain}

\end{document}